\documentclass[11pt]{amsart}

\usepackage[hmargin=0.8in,height=9in]{geometry}
\usepackage{amssymb,amsthm, times, enumerate}
\usepackage{delarray,verbatim}
\usepackage{natbib}
\usepackage{bm}

\usepackage{amssymb}
\usepackage{amsthm}
\usepackage{mathrsfs}

\usepackage{ifpdf}
\ifpdf
\usepackage[pdftex]{graphicx}
 \else
\usepackage[dvips]{graphicx}
 \fi

\usepackage{ifpdf}
\usepackage{color}
\definecolor{webgreen}{rgb}{0,.5,0}
\definecolor{webbrown}{rgb}{.8,0,0}
\definecolor{emphcolor}{rgb}{0.95,0.95,0.95}

\usepackage{hyperref}
\hypersetup{%
          colorlinks=true,
          linkcolor=webbrown,
          filecolor=webbrown,
          citecolor=webgreen,
          breaklinks=true}
\ifpdf \hypersetup{pdftex,
            pdfstartview=FitH, 
            bookmarksopen=true,
            bookmarksnumbered=true
} \else \hypersetup{dvips} \fi

\numberwithin{equation}{section}
\DeclareMathOperator*{\argmin}{\arg\min}
\DeclareMathOperator*{\argmax}{\arg\max}

\newtheorem{theorem}{Theorem}[section]
\newtheorem{proposition}{Proposition}[section]

\newtheorem{remark}{Remark}[section]
\newtheorem{lemma}{Lemma}[section]

\newtheorem{assumption}{Assumption}[section]
\newtheorem{definition}{Definition}[section]
\newtheorem{problem}[theorem]{Problem}
\newtheorem{condition}[theorem]{Condition}

\DeclareMathOperator*{\rqliminf}{\mathrm{r-quick}-\liminf}

\numberwithin{remark}{section} \numberwithin{proposition}{section}
\numberwithin{corollary}{section}
\newcommand{\E}{\mathbb{E}}  
\newcommand{\F}{\mathcal{F}} 
\newcommand{\Fb}{\mathbb{F}} 
\renewcommand{\P}{\mathbb{P}}  
\newcommand{\R}{\mathbb{R}}  

\newcommand{\Y}{\mathcal{Y}}
\newcommand{\Ec}{\mathcal{E}}  
\newcommand{\M}{\mathcal{M}}
\newcommand{\diff}{{\rm d}}

\title[Asymptotic theory of detection and identification]{Asymptotic theory of sequential detection and identification in the hidden Markov models} 
\author[S. Dayanik]{Savas Dayanik$^\dag$\,}\thanks{$\dag$ \tiny{Bilkent University, Departments of Industrial Engineering and
  Mathematics, Bilkent 06800, Ankara, TURKEY. Email: \mbox{sdayanik@bilkent.edu.tr}}}

\author[K. Yamazaki]{\,Kazutoshi Yamazaki$^\ddag$\,}\thanks{$\ddag$\,\tiny{Department of Mathematics,
Faculty of Engineering Science, Kansai University, 3-3-35 Yamate-cho, Suita-shi, Osaka 564-8680, Japan. Email: \mbox{kyamazak@kansai-u.ac.jp}} }

\begin{document}
\maketitle
\date{\today}

\begin{abstract}
  We consider a unified framework of sequential change-point detection
  and hypothesis testing modeled by means of hidden Markov chains.
  One observes a sequence of random variables whose distributions are
  functionals of a hidden Markov chain. The objective is to detect
  quickly the event that the hidden Markov chain leaves a certain set
  of states, and to identify accurately the class of states into which
  it is absorbed.  We propose computationally tractable sequential
  detection and identification strategies and obtain sufficient
  conditions for the asymptotic optimality in two Bayesian
  formulations.  Numerical examples are provided to confirm the
  asymptotic optimality and to examine the rate of convergence.

%
%
%
%
%
%
%
\end{abstract}

\section{Introduction}
The joint problem of sequential change-point detection and hypothesis
testing is generalized in terms of hidden Markov chains.  One observes
a sequence of random variables whose distributions are functionals of
a hidden Markov chain. The objective is to detect as quickly as
possible the disorder, described by the event that the hidden Markov
chain leaves a certain set of states, and to identify accurately its
cause, represented by the class of states into which the Markov chain
is absorbed.  The problem reduces to solving the trade-off between the
expected detection delay and the false alarm and misdiagnosis
probabilities.  A Bayesian formulation of this hidden Markov model has
been proposed by \cite{Dayanik2009}.  It greatly generalizes the
classical models, encompassing change-point detection, sequential
hypothesis testing as well as their joint problem as in
\cite{Dayanik2006}.

There are mainly two directions of research in the Bayesian
formulation.  One direction is to find the means to calculate an
optimal solution, while the other direction is to design
asymptotically optimal solutions that are easy to calculate and
implement. In the first direction, the problem can typically be
expressed in terms of optimal stopping of the posterior probability
process of each alternative hypothesis.  However, there are only a
very few examples that admit analytical solutions, and in practice one
needs to rely on numerical approximations, for example, via value
iteration in combination with discretization of the space of the
posterior probability process.  The computational burden and
nontrivial computer representation of the optimal solution hinder the
application of the findings of this first direction in practice.  The
second direction pursues a strategy that provides simple and scalable
implementation, but gives only near-optimal solutions.  The asymptotic
optimality as a certain parameter of the problem approaches to an
ideal value is commonly used as a proxy for the near-optimality.

Asymptotically optimal strategies are in most cases derived via the
renewal theory.  In the sequential (multiple) hypothesis testing with
i.i.d.\ observations, the log-likelihood ratio (LLR) processes become
conditionally random walks.  By utilizing the ordinary renewal theory,
one can approximate the asymptotic behaviors of the expected sample
size and the misidentification costs; see, for example,
\cite{Baum_Veeravalli1994}.

On the other hand, when the observed random variables are not i.i.d.\
or when the change-point is not geometrically distributed, the
asymptotic optimality is in general not guaranteed; instead, the
existing literature typically shows that the \emph{$r$-quick convergence} of
\cite{Lai1977} of a certain LLR process is a sufficient condition for
asymptotic optimality. \cite{Dragalin_Tartakovsky_Veeravalli1999}
show, under the assumption on the $r$-quick convergence, the
asymptotic optimality of the multihypothesis sequential probability
ratio test (MSPRT) in the non-i.i.d.\ case of sequential multiple
hypothesis testing. \cite{Dragalin_Tartakovsky_Veeravalli2000} further
obtain higher-order approximations by taking into account the
overshoots at up-crossing times of LLR processes.  As for the
change-point detection, \cite{MR2144868} consider the non-i.i.d.\ case
and show the asymptotic optimality of the Shiryaev procedure under the
$r$-quick convergence.  Its continuous-time version is studied by
\cite{Baron_Tartakovsky2006}.

\cite{Dayanik_2012} obtained asymptotically optimal strategies for the
joint problem of change-point detection and sequential hypothesis
testing, showing that the $r$-quick convergence is again a sufficient
condition for asymptotic optimality.  The hidden Markov model is its
generalization, and to the best of our knowledge, its asymptotic
analysis has not been conducted elsewhere. For a comprehensive account
on both analytical and asymptotic optimality of the change-point
detection and sequential hypothesis testing, we refer the reader to
\cite{MR2966314}.

%

This paper gives an asymptotic analysis of the hidden Markov model and
derives asymptotically optimal strategies, focusing on the following
two Bayesian formulations:
\begin{enumerate}
\item In the \emph{minimum Bayes risk formulation}, one minimizes a
  Bayes risk which is the sum of the expected detection delay time and
  the false alarm and misdiagnosis probabilities.
\item In the \emph{Bayesian fixed-error-probability formulation}, one
  minimizes the expected detection delay time subject to some small
  upper bounds on the false alarm and misdiagnosis probabilities.
\end{enumerate}
The optimal strategy of the former has been derived by
\cite{Dayanik2009}.  The latter is usually solved by means of its
Lagrange relaxation, which turns out to be a minimum Bayes risk
problem where the costs are the Lagrange multipliers (or shadow
prices) of the constraints on the false alarm and misdiagnosis
probabilities.  In theory, by employing a hidden Markov chain of an
arbitrary number of states, one can achieve a wide range of realistic
models.  Unfortunately, however, the implementation is computationally
feasible only for simple cases.  The problem dimension is proportional
to the number of states of the Markov chain, and the computation
complexity increases exponentially fast.  This hinders the
applications of the hidden Markov model; in practice, obtaining exact
optimal strategies are still limited to simple and classical examples.

We propose simple and asymptotically optimal strategies for both the
minimum Bayes risk formulation and the Bayesian
fixed-error-probability formulation.  The asymptotic analysis is
similar for both formulations and can be conducted almost
simultaneously.  Similarly to \cite{Dayanik_2012} and to the
non-i.i.d.\ cases of change-point detection and sequential hypothesis
testing as reviewed above, we show that the r-quick convergence for an
appropriate choice of the LLR processes is a sufficient condition for
asymptotic optimality.  We also show in certain cases that the limit
can be analytically derived in terms of the Kullback-Leibler
divergence, and under some conditions higher-order convergence can be
attained using nonlinear renewal theory, which was pioneered by
\cite{Woodroofe1982} and \cite{MR799155}.  Through a sequence of
numerical experiments, we further acknowledge the convergence results
of the LLR processes and the asymptotic optimality of the proposed
strategies.






The remainder of the paper is organized as follows.  In Section
\ref{section_model}, we define the two Bayesian formulations and
review \cite{Dayanik2009}.  In Section \ref{section_new_strategy}, we
propose our strategies and derive sufficient conditions for
asymptotic optimality in terms of the $r$-quick convergence of the LLR
processes.  In Section \ref{section_convergence_results}, we present
examples where the limits of the LLR processes can be analytically
obtained via the Kullback-Leibler divergence.  Section
\ref{section_numerics} concludes the paper with numerical results.

\section{Problem Formulations} \label{section_model} Consider a
probability space $(\Omega, \F, \P)$ hosting a time-homogeneous Markov
chain $Y = (Y_n)_{n \geq 0}$ with some finite state space
$\mathcal{Y}$, initial state distribution $\eta= \{ \eta(y) \in [0,1],
y \in \mathcal{Y}\}$, and one-step transition matrix $P = \{ P(y, y')
\in [0,1], y, y' \in \mathcal{Y} \}$. Suppose that
$\mathcal{Y}_1,\ldots, \mathcal{Y}_M$ are $M$ closed (but not
necessarily irreducible) mutually disjoint subsets of the state space
$\mathcal{Y}$, and let $\mathcal{Y}_0 := \mathcal{Y} \setminus
\bigcup_{k=1}^M \mathcal{Y}_k$.  In other words, $\Y_0$ is transient
and the Markov chain $Y$ eventually gets absorbed into one of the $M$
closed sets. Let us define
\begin{align*}
  \theta := \min \left\{ t \geq 0: Y_t \notin \mathcal{Y}_0 \right\}
  \quad \textrm{and} \quad \mu := \arg \left\{ 1 \leq j \leq M:
    Y_\theta \in \mathcal{Y}_j \right\}
\end{align*}
as the absorption time and the closed set that absorbs $Y$,
respectively.  Here because $\Y_0$ is transient (i.e.\ $\theta <
\infty$ a.s.), $\mu$ is well-defined. We also define $\M := \left\{
  1,\ldots,M \right\}$ and $\M_0 := \M \cup \{0\}$.

\begin{figure}[t!]
  \centering
  \includegraphics{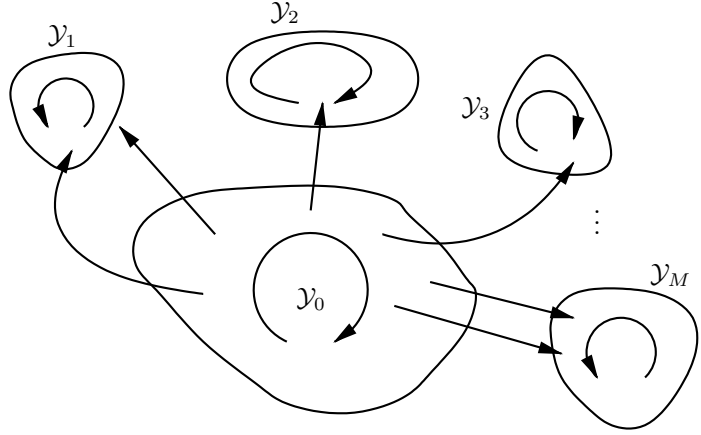}
  \caption{The partition of the state space of the hidden Markov
    model. The problem is to detect the exit time $\theta$ of the
    unobserved $Y$ from $\mathcal{Y}_0$ and identify the index $\mu$
    of the class $\mathcal{Y}_\mu$ into which $Y$ is eventually
    absorbed based only on the observations $X$ modulated by $Y$.}
  \label{hmm}
\end{figure}

The Markov chain $Y$ can be indirectly observed by another stochastic
process $X = (X_n)_{n \geq 1}$ defined on the same probability space
$(\Omega, \F, \P)$.  We assume there exists a set of probability
measures $\left\{ \P (y, \diff x); y \in \Y \right\}$ defined on some
common measurable space $(E,\mathcal{E})$ such that
\begin{align*}
  \P \left\{ Y_0 = y_0, \ldots, Y_t = y_t, X_1 \in E_1, \ldots, X_t
    \in E_t \right\} = \eta (y_0) \prod_{n=1}^t P(y_{n-1},y_n) \P
  (y_n, E_n)
\end{align*}
for every $(y_n)_{0 \leq n \leq t} \in \Y^{t+1}, (E_n)_{1 \leq n \leq
  t} \in \mathcal{E}^t, t \geq 1$.  For every $y \in \Y$, we assume
that $\P (y, \diff x)$ admits a density function $f(y, x)$ with
respect to some $\sigma$-finite measure $m$ on $(E,\Ec)$;
namely,\begin{align*} f (y,x) m (\diff x) = \P (y, \diff x).
\end{align*}

Let $\Fb = (\F_n)_{n \geq 0}$ denote the filtration generated by the
stochastic process $X$; namely,
\begin{align*}
  \F_0 = \{\varnothing, \Omega \} \quad \textrm{and} \quad \F_n =
  \sigma(X_1,\dots,X_n), \quad n \geq 1.
\end{align*}
A (sequential decision) strategy $(\tau,d)$ is a pair of an $\Fb
$-stopping time $\tau$ (in short, $\tau \in \Fb$) and a random
variable $d: \Omega \rightarrow \M$ that is measurable with respect to
the observation history $\F_\tau$ up to the stopping time $\tau$
(namely, $d \in \F_{\tau}$).  Let
\begin{align*}
  \Delta := \left\{ (\tau,d) : \tau \in \Fb \; \text{and} \; d \in
    \F_\tau \; \text{is an $\M$-valued random variable}\right\}
\end{align*}
be the set of strategies.

The objective is to obtain a strategy $(\tau,d)$ so as to minimize the
expected detection delay (EDD)
\begin{align}
  D^{(c,m)}(\tau) := \E \Big[ \Big(\sum_{t=0}^\infty c(Y_t) 1_{\{t <
    \tau\}} \Big)^m \Big]\label{def_edd}
\end{align}for some $m \geq 1$ and deterministic nonnegative and
bounded function $c: \mathcal{Y} \rightarrow [0,\infty)$, as well as
the terminal decision losses (TDL's)
\begin{align}
  R_{yi}(\tau,d) := \P \left\{ d=i, Y_\tau = y, \tau < \infty
  \right\}, \quad i \in \M, \, y \in \Y \setminus
  \Y_i. \label{def_tdc}
\end{align}
The Bayes risk is a linear combination of all of these losses,
\begin{align}
  u^{(c,a,m)}(\tau,d) := D^{(c,m)}(\tau) + \sum_{i \in \mathcal{M}}
  \sum_{y \in \mathcal{Y} \setminus \Y_i} a_{yi} R_{yi}
  (\tau,d) \label{bayes_risk}
\end{align}
for some $m \geq 1$, $c$, and a set of strictly positive constants $a=
(a_{yi})_{i \in \M, y \in \Y \setminus \Y_i }$.  In \eqref{def_edd},
while it is natural to assume $c(y)=0$ for $y \in \Y_0$, we allow
$c(y)$ to take any nonnegative values for $y \in \mathcal{Y}_0$.  On
the other hand, in \eqref{def_tdc} and \eqref{bayes_risk}, we assume
that any correct terminal decision (i.e., $\{d=i, Y_\tau \in \Y_i,
\tau < \infty \}$) is not penalized because otherwise the terminal
decision loss \eqref{def_tdc} cannot be bounded by small numbers and
Problem \ref{problem_invariant} below does not make sense.
  

\begin{problem}[Minimum Bayes risk
  formulation] \label{problem_bayes_risk} Fix $m \geq 1$, $c$, and a
  set of strictly positive constants $a= (a_{yi})_{i \in \M, y \in \Y
    \setminus \Y_i }$, we want to calculate the minimum Bayes risk
  \begin{align*}
    \inf_{(\tau,d) \in \Delta} u^{(c,a,m)}(\tau,d)
  \end{align*}
  and find a strategy $(\tau^*,d^*)$ that attains it, if such a
  strategy exists.
\end{problem}

\begin{problem}[Bayesian fixed-error probability
  formulation] \label{problem_invariant} Fix $m\geq 1$, $c$, and a set
  of strictly positive constants $\overline{R} =
  (\overline{R}_{yi})_{i \in \M, y \in \mathcal{Y} \setminus \Y_i}$,
  we want to calculate the minimum EDD
  \begin{align*}
    \inf_{(\tau,d) \in \Delta(\overline{R})} D^{(c,m)}(\tau)
  \end{align*}
  where
  \begin{align*}
    \Delta(\overline{R}) := \left\{ (\tau,d) \in \Delta:
      R_{yi}(\tau,d) \leq \overline{R}_{yi}, \; i \in \M, y \in \Y
      \setminus \Y_i \right\},
  \end{align*}
  and find a strategy $(\tau^*,d^*)\in \Delta (\overline{R})$ that
  attains it, if such a strategy exists.
\end{problem}

For every $i \in \M$, define
\begin{align*}
  \widetilde{R}_{ji}(\tau,d) := \sum_{y \in \mathcal{Y}_j}
  R_{yi}(\tau,d) = \left\{ \begin{array}{ll} \P \left\{ d=i, \tau <
        \theta \right\} & j = 0, \\ \P \left\{ d=i, \mu =j, \theta
        \leq \tau < \infty \right\}, & j \in \M \setminus
      \{i\}. \end{array} \right.
\end{align*}

\begin{remark} \label{remark_set} Fix a set of positive constants
  $\overline{R}$. We have
  \begin{align*}
    \Delta(\overline{R}) &\subset \Big\{ (\tau,d) \in \Delta:
    \widetilde{R}_{ji}(\tau,d) \leq
    \sum_{y \in \Y_j }\overline{R}_{yi}, \; i \in \M, j \in \M_0 \setminus \{ i \} \Big\}  =: \overline{\Delta} (\overline{R}),\\
    \Delta(\overline{R}) &\supset \Big\{ (\tau,d) \in \Delta:
    \widetilde{R}_{ji}(\tau,d) \leq \min_{y \in \Y_j
    }\overline{R}_{yi}, \; i \in \M, j \in \M_0 \setminus \{ i \}
    \Big\} =: \underline{\Delta} (\overline{R}).
  \end{align*}
\end{remark}

In our analysis, we will need to reformulate the problem in terms of
the conditional probabilities
\begin{align*}
  \P_i \left\{X_1 \in E_1,...,X_n \in E_n\right\} &:= \P \left\{\left.
      X_1 \in E_1,...,X_n \in E_n \right| \mu = i\right\},   \\
  \P_i^{(t)} \left\{X_1 \in E_1,...,X_n \in E_n\right\} &:= \P
  \left\{\left. X_1 \in E_1,...,X_n \in E_n \right| \mu = i, \theta =
    t\right\}, \quad t \geq 0,
\end{align*}
defined for every $i \in \M$, $n \geq 1$ and $(E_1 \times \cdots
\times E_n) \in \Ec^n$. Let $\E_i$ and $\E_i^{(t)}$ be the
expectations with respect to $\P_i$ and $\P_i^{(t)}$, respectively. We
also let the unconditional probability that $Y$ is absorbed by
$\mathcal{Y}_i$ be
\begin{align*}
  \nu_i &:= \P \left\{ \mu = i\right\}, \quad i \in \M.
\end{align*}
Because $\Y_0$ is transient, we must have $\sum_{i \in \M} \nu_i =1$.
Without loss of generality, we can assume $\nu_i > 0$ for any $i \in
\M$ because otherwise we can disregard $\Y_i$ and consider the Markov
chain on $\Y\setminus \Y_i$.

In terms of those conditional probabilities, we have $D^{(c,m)}(\tau)
= \sum_{i \in \M} \nu_i D_i^{(c,m)}(\tau)$, where
\begin{align*}
  D_i^{(c,m)}(\tau) := \E_i \Big[\Big(\sum_{t=0}^\infty c(Y_t) 1_{\{t
    < \tau\}} \Big)^m \Big], \quad i \in \M, \; (\tau,d) \in \Delta.
\end{align*}
We decompose the Bayes risk such that
\begin{align*}
  u^{(c,a,m)}(\tau,d) = \sum_{i \in \M} \nu_i u_i^{(c,a,m)}(\tau,d)
\end{align*}
where
\begin{align}
  u_i^{(c,a,m)}(\tau,d) &:= D_i^{(c,m)} (\tau) + R_i^{(a)} (\tau,d), \label{bayes_risk_i} \\
  R_i^{(a)} (\tau,d) &:= \frac 1 {\nu_i} \sum_{y \in \mathcal{Y}
    \setminus \Y_i} a_{yi} R_{yi} (\tau,d) \label{r_i_a}
\end{align}
for every $(\tau,d) \in \Delta$. In particular, with $a_{yi} = 1$ for
all $y \in \Y \backslash \Y_i$,
\begin{align}
  R_i^{(1)} (\tau,d) := \frac 1 {\nu_i} \sum_{y \in \mathcal{Y}
    \setminus \Y_i} R_{yi} (\tau,d) = \frac 1 {\nu_i} \sum_{j \in \M_0
    \setminus \{i\}} \widetilde{R}_{ji} (\tau,d), \quad (\tau,d) \in
  \Delta. \label{r_i_1_temp}
\end{align}

\section{Asymptotically Optimal Strategies} \label{section_new_strategy}
We now introduce two strategies. 
The first strategy triggers an alarm when the posterior probability of the event that $Y$ has been absorbed by a certain closed set  exceeds some threshold for the first time, and will be later proposed as an asymptotically optimal solution for Problem \ref{problem_bayes_risk}. The second strategy is its variant expressed in terms of the log-likelihood ratio (LLR) processes and will be proposed as an asymptotically optimal solution for Problem \ref{problem_invariant}. 

For all $y \in \Y$, let $(\Pi_n(y))_{n \geq 0}$ be the
posterior probability process defined by
\begin{align*}
  \Pi_n(y) :=  \P \left\{ \left. Y_n = y \right| \F_n \right\}, \quad y \in \Y.
\end{align*}
Then $\Pi_0(y) = \eta(y)$, $y \in \Y$, and  for $n \geq 1$
\begin{align*}
\Pi_n(y) = \frac {\alpha_n(X_1,\ldots,X_n,y)} {\sum_{y' \in \mathcal{Y}}\alpha_n(X_1,\ldots,X_n,y')}
\end{align*}
where
\begin{align}
\alpha_n (x_1,\ldots, x_n,y) := \sum_{(y_0,\ldots, y_{n-1}) \in \mathcal{Y}^n} \left( \eta(y_0) \prod_{k=1}^{n-1} P(y_{k-1},y_k) f(y_k,x_k)\right) P(y_{n-1},y) f(y,x_n); \label{def_alpha}
\end{align}
see \cite{Dayanik2009}.   Also define
\begin{align*}
\widetilde{\Pi}_n^{(i)} := \P \left\{ \left. Y_n \in \Y_i \right| \F_n \right\} = \left\{ \begin{array}{ll} \P \left\{ \left. \theta > n \right| \F_n \right\}, & i=0 \\ \P \left\{ \left. \theta
      \leq n, \mu=i \right| \F_n \right\}, & i \in \M \end{array}\right\}.
\end{align*}
Then $\widetilde{\Pi}_0^{(i)} = \sum_{y \in \Y_i} \eta(y)$, $i \in \M_0$, and for $n \geq 1$
\begin{align*}
\widetilde{\Pi}_n^{(i)} = \sum_{y \in \mathcal{Y}_i}\Pi_n(y) = \frac {\widetilde{\alpha}_n^{(i)}(X_1,\ldots,X_n) } {\sum_{j \in \M_0}\widetilde{\alpha}_n^{(j)}(X_1,\ldots,X_n) },
\end{align*}
where
\begin{align}
\widetilde{\alpha}_n^{(i)}(x_1,\ldots,x_n) := \sum_{y \in \mathcal{Y}_i} \alpha_n(x_1,\ldots,x_n,y), \quad i \in \M_0, \, (x_1,\ldots, x_n) \in E^n. \label{def_alpha_sum}
\end{align}
For the rest of the paper, we use the short-hand notations: $\widetilde{\alpha}_n^{(i)}:= \widetilde{\alpha}_n^{(i)}(X_1,\ldots,X_n)$ for any $n \geq 1$ and $i \in \M_0$.

\begin{assumption} \label{finiteness_ratio}
For every $y,z \in \Y$, we assume $0 < f(y,X_1)/f(z,X_1) < \infty$ a.s.
This implies that $0< \widetilde{\Pi}_n^{(i)} <
  1$ a.s.\ for every finite $n \geq 1$ and $i\in \M$.
\end{assumption}

Let $\Lambda(i,j) = \left( \Lambda_n(i,j) \right)_{n \geq 1}$ be the LLR processes;
\begin{align}
  \Lambda_n(i,j) := \log  \frac {\widetilde{\Pi}_n^{(i)}}{\widetilde{\Pi}_n^{(j)}} = \log   \frac {\widetilde{\alpha}_n^{(i)}} {\widetilde{\alpha}_n^{(j)}}, \quad n \geq 1, \; i \in \M, \; j \in \M_0\setminus\{i\}. \label{eq:log-likelihood-ratio-processes}
\end{align}
%

\begin{definition}[$(\tau_A,d_A)$-strategy for the minimum Bayes risk formulation]
Fix a set of strictly positive constants $A = (A_i)_{i \in \M}$, define strategy $(\tau_A,d_A)$ by
\begin{align}
 \tau_A = \min_{i \in \M}\; \tau_A^{(i)} \quad \textrm{and} \quad d_A  \in \argmax_{i \in \M} \widetilde{\Pi}_{\tau_A}^{(i)}, \label{def_strategy_a}
\end{align}
where
\begin{align}
\tau_A^{(i)} := \inf \left\{ n \geq 1: \widetilde{\Pi}_n^{(i)} > \frac 1 {1+A_i}\right\}, \quad i \in \M.  \label{def_tau_a_i}
\end{align}  
\end{definition}

Define the logarithm of the odds-ratio process 
\begin{align}
\Phi_n^{(i)} := \log \frac {\widetilde{\Pi}_n^{(i)}} {1-\widetilde{\Pi}_n^{(i)}}= - \log \Big[ \sum_{j \in \M_0 \setminus \{i\}} \exp \left(-\Lambda_n(i,j)
  \right) \Big] = \log \frac {\widetilde{\alpha}_n^{(i)}} {\sum_{j \in \M_0 \setminus \{i\}}\widetilde{\alpha}_n^{(j)}}\label{identity_phi}, \quad i \in \M, \; n \geq 1. 
\end{align}
Then, (\ref{def_tau_a_i}) can be rewritten as
\begin{align*}
  \tau_A^{(i)} = \inf \left\{ n \geq 1 : \frac
    {1-\widetilde{\Pi}_n^{(i)}}{\widetilde{\Pi}_n^{(i)}} < A_i \right\} =  \inf \left\{n \geq 1: \Phi_n^{(i)} > - \log A_i \right\}, \quad i \in \M.
\end{align*}
\begin{definition}[$(\upsilon_B,d_B)$-strategy for the Bayesian fixed-error-probability formulation]
Fix a set of strictly positive constants $B = (B_{ij})_{i \in \M, \; j \in \M_0 \setminus \{i\}}$, define
\begin{align*}
\upsilon_B := \min_{i \in \M} \upsilon_B^{(i)} \quad \textrm{and} \quad d_B \in \arg \min_{i \in \M} \upsilon_B^{(i)}
\end{align*}
where
\begin{align}
\upsilon_B^{(i)} := \inf \left\{ n \geq 1: \Lambda_n(i,j) > - \log B_{ij}  \textrm{ for every $j \in \M_0 \setminus \{i\}$}\right\}, \quad i \in \M. \label{definition_upsilon}
\end{align}
\end{definition}

Fix $i \in \M$. Define
\begin{align*}
\overline{B}_i := \max_{j \in \M_0 \setminus \{i\}} B_{ij} \quad \textrm{and} \quad \underline{B}_i := \min_{j \in \M_0 \setminus \{i\}} B_{ij},
\end{align*}
and the minimum of the LLR processes,
\begin{align*}
\Psi^{(i)}_n := \min_{j \in \M_0 \setminus \{i\}} \Lambda_n(i,j), \quad n \geq 1. 
\end{align*}
Then we have 
\begin{align*}
\underline{\upsilon}_B^{(i)} \leq \upsilon_B^{(i)} \leq \overline{\upsilon}_B^{(i)}, 
\end{align*}
 where
\begin{align*}
\underline{\upsilon}_B^{(i)} &:= \inf \left\{ n \geq 1: \Psi^{(i)}_n > - \log \overline{B}_{i} \right\},  \\
\overline{\upsilon}_B^{(i)} &:= \inf \left\{ n \geq 1: \Psi^{(i)}_n > - \log \underline{B}_{i} \right\}. 
\end{align*}
Notice by (\ref{identity_phi}) that $\Phi_n^{(i)} \leq \Lambda_n(i,j)$ for every $n \geq 1$ and $j \in \M_0 \setminus \{i\}$, and hence
\begin{align}
\Psi^{(i)}_n \geq \Phi_n^{(i)}, \quad n \geq 1. \label{psi_greater_than_phi}
\end{align}

We will show that, by adjusting the values of $A$ and $B$, the strategy $(\tau_A,d_A)$ is asymptotically optimal in Problem \ref{problem_bayes_risk} as  
\begin{align*}
\|c\| := \max_{y \in \Y} c (y) \downarrow 0
\end{align*}
for fixed $a$, and the strategy $(\upsilon_B,d_B)$ is asymptotically optimal in Problem \ref{problem_invariant} as 
\begin{align*}
\|\overline{R}\| := \max_{i \in \M, \; y \in \Y \setminus \Y_i } \overline{R}_{yi} \downarrow 0
\end{align*}  
for fixed $c$.
For the latter, we assume that, in taking limits, $\overline{R}_i := (\overline{R}_{yi})_{y \in \Y \backslash \Y_i}$ satisfy
\begin{align}
\frac {\min_{y \in \Y \setminus \Y_i} \overline{R}_{yi}} {\max_{y \in \Y \setminus \Y_i} \overline{R}_{yi}} > \beta_i, \quad i \in \M, \label{bounds_r_ratio}
\end{align}
for some strictly positive constants $(\beta_i)_{i \in \M}$.  This limit mode will still be denoted by ``$\|\overline{R}\|\downarrow 0$''
for brevity.


We will find functions $A(c)$ and $B(\overline{R})$ so that 
\begin{align}
u^{(c,a,m)}(\tau_{A(c)},d_{A(c)}) &\sim \inf_{(\tau,d) \in \Delta}
u^{(c,a,m)}(\tau,d) \quad \textrm{as }  \| c \| \downarrow 0, \label{optimality_bayes_risk} \\
D^{(c,m)}(\upsilon_{B(\overline{R})}) &\sim \inf_{(\tau,d) \in \Delta (\overline{R})} D^{(c,m)} (\tau) \quad \textrm{as }  \| \overline{R} \| \downarrow 0, \label{optimality_invariant}
\end{align}
where
\begin{align*}
x_{\gamma} \sim y_\gamma \textrm{ as } \gamma \rightarrow \gamma_0 \Longleftrightarrow \lim_{\gamma \rightarrow \gamma_0} \frac {x_\gamma} {y_\gamma} = 1.
\end{align*}
In fact, we will obtain results stronger than (\ref{optimality_bayes_risk}) and (\ref{optimality_invariant}); we will show
\begin{align}
u_i^{(c,a,m)}(\tau_{A(c)},d_{A(c)}) &\sim \inf_{(\tau,d) \in \Delta}
u_i^{(c,a,m)}(\tau,d) \quad \textrm{as } \|c\| \downarrow 0, \label{optimality_bayes_risk_strong} \\
D^{(c,m)}_i(\upsilon_{B(\overline{R})}) &\sim \inf_{(\tau,d) \in \Delta (\overline{R})} D^{(c,m)}_i (\tau) \quad \textrm{as }  \| \overline{R} \| \downarrow 0, \label{optimality_invariant_strong}
\end{align}
for every $i \in \M$.

\subsection{Convergence of terminal decision losses and detection delay}

As $c$ and $\overline{R}$ decrease in Problems \ref{problem_bayes_risk} and \ref{problem_invariant}, respectively, the optimal stopping regions shrink and one should expect to wait longer. In Problem \ref{problem_bayes_risk}, when the unit sampling cost is small, one should take advantage of it and sample more.   In Problem \ref{problem_invariant}, when the upper bounds on the TDL's are small, one expects to wait longer to collect more information in order to satisfy the constraints. 
Moreover, the size of the stopping regions for $(\tau_A,d_A)$ and $(\upsilon_B,d_B)$ decrease monotonically as $A$ and $B$ decrease. Therefore, functions $A(c)$ and $B(\overline{R})$ should be monotonically decreasing as $c$ and $\overline{R}$ decrease, respectively. We explore the asymptotic behaviors of the EDD and the TDL as $A \downarrow 0$ and $B \downarrow 0$.

Define
\begin{align*}
\| A \| := \max_{i \in \M} A_i \quad \textrm{and} \quad \| B \| := \max_{i \in \M, \, j \in \M_0 \setminus \{i\}} B_{ij}. 
\end{align*}
 Moreover, assume, while taking limits $\|B\| \downarrow 0$, that the ratio $\underline{B}_i/\overline{B}_i$ for every $i \in \M$ is bounded from below by some strictly positive number so that it is consistent with how $\|\overline{R}\|$ decreases to $0$ as we assumed in (\ref{bounds_r_ratio}). 

We first obtain bounds on the TDL's that are shown to converge to zero in the limit.
The LLR processes can be used as Radon-Nikodym derivatives to change measures as the following lemma shows.  The proof is the same as Lemma 2.3 of  \cite{Dayanik_2012}, and hence we omit it.
\begin{lemma}[Changing Measures] 
  \label{lemma_changing_measure} 
  Fix $i \in \M$, an $\Fb $-stopping time $\tau$, and an
  $\F_\tau$-measurable event $F$.  We have
  \begin{align*}
    \P \left(F \cap \left\{\mu =j, \theta \leq \tau < \infty\right\}\right) &=
    \nu_i\, \E_i \left[ 1_{F \cap \{\theta \leq \tau <
        \infty\}} e^{-\Lambda_\tau(i,j)}\right], \quad j \in \M\setminus\{i\}, \\
 \P \left( F \cap \left\{ \tau < \theta
      \right\}\right) &= \nu_i\, \E_i \left[ 1_{F \cap \left\{ \theta \leq
          \tau < \infty \right\}} e^{-\Lambda_\tau(i,0)}\right].
  \end{align*}
\end{lemma}

The next proposition can be obtained by setting $F:= \{d=i\} \in \F_{\tau}$ in Lemma \ref{lemma_changing_measure}.
\begin{proposition}\label{corollary_measure_change} For every strategy $(\tau,d) \in \Delta$, we have
\begin{align*}
    \widetilde{R}_{ji}(\tau,d) &= \nu_i\, \E_i \left[ 1_{\{d=i, \; \theta \leq
        \tau < \infty\}} e^{-\Lambda_\tau(i,j)}\right], \quad i \in \M, \; j \in
   \M_0\setminus \{i\}.
\end{align*}
\end{proposition}
In particular, \eqref{r_i_1_temp} can be rewritten 
\begin{align}
R_i^{(1)}(\tau,d) &= \E_i \Big[ 1_{\left\{ d=i, \; \theta
          \leq \tau < \infty\right\}} \sum_{j \in \M_0 \setminus \{i\}} e^{-\Lambda_{\tau} (i,j)} \Big], \quad i \in \M, \; (\tau,d) \in \Delta. \label{r_i_1}
\end{align}

\begin{remark} \label{remark_bounds_tdl}
Fix $i \in \M$. Let 
\begin{align*}
\overline{a}_i := \max_{y \in \Y \setminus \Y_i} a_{yi}.
\end{align*}
By \eqref{r_i_a}, \eqref{r_i_1_temp} and \eqref{r_i_1}, 
\begin{align*}
R_i^{(a)} (\tau,d) \leq \frac 1 {\nu_i} \sum_{j \in \M_0 \setminus \{i\}} (\max_{y \in \Y_j} a_{yi}) \widetilde{R}_{ji} (\tau,d) \leq \overline{a}_i R_{i}^{(1)} (\tau,d) \leq \overline{a}_i \E_i \Big[ 1_{\left\{ d=i, \; \theta
          \leq \tau < \infty\right\}} \sum_{j \in \M_0 \setminus \{i\}} e^{-\Lambda_{\tau} (i,j)} \Big].
\end{align*}
\end{remark}

With this remark, we attain a slight modification of Proposition 3.4 of \cite{Dayanik_2012}.
\begin{proposition}[Bounds on the TDL] \label{corollary_bound_kappa_i_and_r} We can obtain the following bounds on the TDL's.
\begin{itemize}
\item[(i)] For every fixed $A = (A_i)_{i \in \M}$ and $a = (a_{yi})_{i \in \M, y \in \mathcal{Y} \setminus \Y_i}$, we have
\begin{align*}
R_i^{(a)}(\tau_A,d_A) &\leq \overline{a}_i A_i, \quad i \in \M. 
\end{align*}
\item[(ii)] For every $B = (B_{ij})_{i \in \M, j \in \M \setminus\{i\}}$, we have
\begin{align*}
\widetilde{R}_{ji} (\upsilon_B, d_B) &\leq \nu_i B_{ij}, \quad i \in \M, \, j \in \M_0 \setminus\{i\}.
\end{align*}
\end{itemize}
\end{proposition}

Using the bounds in Proposition \ref{corollary_bound_kappa_i_and_r} and Remark \ref{remark_set}, we can obtain feasible strategies by choosing the values of $A$ and $B$ accordingly.

\begin{proposition}[Feasible Strategies] \label{proposition_feasible_strategy}
Fix a set of strictly positive constants $\overline{R} = (R_{yi})_{i \in \M, y \in \Y \setminus \Y_i}$. If $B_{ij} (\overline{R}) \leq  {\min_{y \in \Y_j} \overline{R}_{yi}/\nu_i} $ for every $i \in \M$ and  $j \in \M_0 \setminus \{i\}$, then $(\upsilon_{B(\overline{R})}, d_{B(\overline{R})}) \in \Delta (\overline{R})$.
\end{proposition}


We now analyze the asymptotic behavior of the detection delay.  Proposition \ref{proposition_infinity_tau_i} below allows us to use $\tau^{(i)}_A \uparrow \infty$ (resp.\ $\upsilon^{(i)}_B \uparrow \infty$) and $A_i \downarrow 0$ (resp.\ $B_i \downarrow 0$) interchangeably for every $i \in \M$.
Its proof is the same as that of Proposition 3.6 of \cite{Dayanik_2012}.
\begin{proposition} \label{proposition_infinity_tau_i} Fix $i
  \in \M$. We have $\P_i$-a.s., 
  \begin{enumerate}
\item[(i)] $\tau_A^{(i)} \to \infty$ as $A_i \downarrow 0$ and $\tau_A \to \infty$ as $\|A\| \downarrow 0$,
\item[(ii)] $\upsilon_B^{(i)} \to \infty$ as $\overline{B}_i \downarrow 0$ and $\upsilon_B \to \infty$ as $\|B\| \downarrow 0$.
  \end{enumerate}
\end{proposition}

The posterior probability process $(\widetilde{\Pi}^{(i)}_n)_{i \in \M_0}$ converges a.s.\ by \cite{Dayanik2009}. Moreover, because the posterior probability of the correct hypothesis should tend to increase in the long run, on the event $\{\mu=i\}$, $i \in \M$, it is expected that $\widetilde{\Pi}_n^{(i)}$ converges to $1$ and that $\widetilde{\Pi}_n^{(j)}$ converges to $0$ for every $j \in \M_0 \setminus \{i\}$ with probability one. This suggests the a.s.-convergence of $\Lambda_n(i,j)$ to infinity given $\mu=i$ for every $j \in \M_0 \setminus \{i\}$.  For the rest of this section, we further assume that the average increment converges to some strictly positive value.
\begin{assumption} \label{condition_a_s_convergence}
For every $i \in \M$, we assume that
\begin{align*}
\Lambda_n(i,j)/n \xrightarrow[n \uparrow \infty]{\P_i-a.s.} l(i,j),
\end{align*}
for some $ l(i,j) \in (0,\infty]$ for every $j \in \M_0 \setminus \{i\}$, and 
\begin{align*}
\min_{j \in \M_0 \setminus \{i\}} l(i,j) < \infty.
\end{align*}
\end{assumption}
This is indeed satisfied in the i.i.d.\ case (\cite{Dayanik_2012}).  In Section \ref{section_convergence_results},  we will show that this is also satisfied in certain more general settings and that the limit can be expressed in terms of the Kullback-Leibler divergence.



Let us fix any $i \in \M$. We show that, for small values of $A$ and
$B$, the stopping times $\tau_A^{(i)}$ and $\upsilon_B^{(i)}$ in \eqref{def_strategy_a} and \eqref{definition_upsilon} are
essentially determined by the process $\Lambda(i,j(i))$, where
\begin{align*}
  j(i) \in \argmin_{j \in \M_0 \setminus \{i\}} l(i,j)
  \quad \text{is any index in $\M_0 \setminus\{i\}$ that attains}\quad
  l(i) := \min_{j \in \M_0 \setminus \{i\}} l(i,j) > 0,
\end{align*}
and $\P_i$-a.s.\ $\Lambda_n(i,j(i))/n \approx \Phi_n^{(i)}/n \approx
\Psi^{(i)}_n/n \approx l(i)$ for sufficiently large $n$ as the
next proposition implies.

\begin{proposition} \label{proposition_asymptotic_phi} For every $i
  \in \M$, we have $\P_i$-a.s.\ (i) $\Phi_n^{(i)}/n \rightarrow
  l(i)$ and (ii) $\Psi_n^{(i)}/n \rightarrow l(i)$ as $n \uparrow
  \infty$.
\end{proposition}



For the proof of Proposition \ref{proposition_asymptotic_phi} above, (ii) follows immediately by Assumption \ref{condition_a_s_convergence} and (i) follows from Lemma \ref{lemma_convergence_min_to_mu} below after
replacing $Y^{(j)}_n$, $\P$, and $(\mu_j)_{j \in \M_0 \setminus \{i\}}$ in the lemma with
$\Lambda_n(i,j)/n$, $\P_i$, and $(l(i,j))_{j \in \M_0 \setminus \{i\}}$, respectively, for
every fixed $i\in \M$.

\begin{lemma}\label{lemma_convergence_min_to_mu} For every $j=1,\dots
  ,m$, let $Y^{(j)}=(Y_n^{(j)})_{n \geq 1}$ be a sequence of random
  variables defined on a common probability space $(\Omega, \Ec,\P)$,
  and suppose that $Y_n^{(j)}$ converges a.s.\ to some constant
  $\mu_j \in [-\infty,\infty]$ for every $j=1,\dots ,m$. Then $- \frac 1 n \log \sum_{j=1}^m
      e^{-n Y_n^{(j)}} \xrightarrow[n \uparrow \infty]{a.s.} \min_{1 \leq k \leq m} \mu_k$.
\end{lemma}
Lemma \ref{lemma_convergence_min_to_mu} is a straightforward extension of Lemma 5.2 of \cite{Baum_Veeravalli1994} and is omitted.

The following lemma can be derived from Proposition \ref{proposition_asymptotic_phi}.  The proof is the same as that of Lemma 3.9 of  \cite{Dayanik_2012}.
\begin{lemma}\label{proposition_stopping_time_asymptotic_lemma}
  For every $i \in \M$ and any $j(i) \in \argmin_{j\in \M_0\setminus
    \{i\}} l(i,j)$, we have $\P_i$-a.s. 
  \begin{align*}
    \textrm{\normalfont (i)} &\quad -\frac{\tau_A^{(i)}}{\log A_i}
    \xrightarrow{A_i \downarrow 0}
    \frac{1}{l(i)}, && \textrm{\normalfont (ii)} \quad
    - \frac{(\tau_A^{(i)}-\theta)_+} {\log A_i} \xrightarrow{A_i
    \downarrow 0} \frac{1}{l(i)},\\
    \textrm{\normalfont (iii)} &\quad -\frac{\upsilon_B^{(i)}}{\log
      B_{ij(i)}} \xrightarrow{\overline{B}_i \downarrow
    0}  \frac{1}{l(i)}, &&
    \textrm{\normalfont (iv)} \quad
    -\frac{(\upsilon_B^{(i)}-\theta)_+}{\log B_{ij(i)}}
    \xrightarrow{\overline{B}_i \downarrow 0}
    \frac{1}{l(i)}.
  \end{align*}
  
\end{lemma}

\begin{remark}
\label{rem:concordance-between-B-values}
We shall assume that $0< B_{ij}<1$ or $-\infty <\log
  B_{ij}<0$ for all $i \in \M$ and $j\in\M_0 \setminus\{i\}$ as we are interested in the limits of
  certain quantities as $\|B\|\downarrow 0$.  This implies
  \begin{align*}
    1=\lim_{\overline{B}_i \downarrow 0} \; \frac{\log B_{ij}}{\log
      \overline{B}_i} = \lim_{\overline{B}_i \downarrow 0} \;
    \frac{\log \underline{B}_i}{\log \overline{B}_i} =
    \lim_{\overline{B}_i \downarrow 0} \; \frac{\log B_{ij}}{\log
      \underline{B}_i} \quad \text{for every $i\in\M$, $j\in
      \M_0\setminus \{i\}$,}
  \end{align*}
  where the last equality follows from the first two equalities.
\end{remark}

For every $i \in \M$, conditionally on $\{Y_0 \in \Y_i\}$, the Markov chain $Y$ always admits a stationary distribution; namely, there exists a unique nonnegative $w_i(y)$, for every $y \in \Y_i$, such that
\begin{align*}
\frac 1 {n+1} \sum_{m=0}^n 1_{\{y\}} (Y_m) \xrightarrow[\P_i-a.s.]{n \uparrow \infty} w_i(y), \quad  \textrm{ on } \{ Y_0 \in \mathcal{Y}_i\};
\end{align*}
see, e.g., \cite{Tijms_2003}.
Then
\begin{align*}
\frac 1 {n+1} \sum_{m=0}^n c(Y_m) \xrightarrow[\P_i-a.s.]{n \uparrow \infty} c_i : = \sum_{y \in \Y_i} c(y) w_i(y), \quad  \textrm{ on } \{ Y_0 \in \mathcal{Y}_i\}.
\end{align*}
This and the a.s.\ finiteness of $\theta$ together with Lemma
\ref{proposition_stopping_time_asymptotic_lemma} prove the next
lemma.
\begin{lemma}
  For every $i \in \M$ and any $j(i) \in \argmin_{j\in \M_0\setminus
    \{i\}} l(i,j)$, we have $\P_i$-a.s. 
  \begin{align*}
    \textrm{\normalfont (i)} &\quad -\frac{\sum_{m=0}^{\tau_A^{(i)}} c(Y_m)}{\log A_i}
    \xrightarrow{A_i \downarrow 0}
    \frac{c_i}{l(i)}, && \textrm{\normalfont (ii)} \quad
    - \frac{\sum_{m=\theta}^{\tau_A^{(i)} \vee \theta} c(Y_m)} {\log A_i} \xrightarrow{A_i
    \downarrow 0} \frac{c_i}{l(i)},\\
    \textrm{\normalfont (iii)} &\quad -\frac{\sum_{m=0}^{v_B^{(i)}} c(Y_m)}{\log
      B_{ij(i)}} \xrightarrow{\overline{B}_i \downarrow
    0}  \frac{c_i}{l(i)}, &&
    \textrm{\normalfont (iv)} \quad
    -\frac{\sum_{m=\theta}^{v_B^{(i)} \vee \theta} c(Y_m)}{\log B_{ij(i)}}
    \xrightarrow{\overline{B}_i \downarrow 0}
    \frac{c_i}{l(i)}.
  \end{align*}
\end{lemma}

Because we want to minimize the $m^{th}$ moment of the detection delay
time for any $m\ge 1$, we will strengthen the convergence
results of Lemma \ref{proposition_stopping_time_asymptotic_lemma}. 
We require Condition \ref{condition_uniformly_integrable} below for some $r \geq m$.
  
  \begin{condition}[Uniform Integrability] \label{condition_uniformly_integrable}
  For given $r \geq 1$, we assume that
  \begin{itemize}
  \item[(i)] $(\tau_A^{(i)}/(-\log A_i))^r_{A_i > 0}$ is $\P_i$-uniformly integrable for every $i \in \M$,
  \item[(ii)] $(\upsilon_B^{(i)}/(-\log B_{ij(i)}))^r_{B_i > 0}$ is $\P_i$-uniformly integrable  for every $i \in \M$.
  \end{itemize}
  \end{condition}
Because $c(\cdot)$ is bounded, this also implies the following.
  \begin{lemma} For every $i \in \M$, we have the followings.
  \begin{itemize}
\item[(i)]Under Condition \ref{condition_uniformly_integrable} (i) for some $r \geq 1$,  $\big( (\sum_{m=0}^{\tau_A^{(i)}} c(Y_m))/(-\log A_i) \big)_{A_i > 0}$ is also $\P_i$-uniformly integrable.
\item[(ii)]Under Condition \ref{condition_uniformly_integrable} (ii)  for some $r \geq 1$,   $\big( (\sum_{m=0}^{\upsilon_B^{(i)}} c(Y_m))/(-\log B_{ij(i)}) \big)_{B_i > 0}$ is also $\P_i$-uniformly integrable.
\end{itemize}
\end{lemma}
Hence, Condition \ref{condition_uniformly_integrable} for some $r \geq
m$ is sufficient for the $L^m$-convergences.  


    \begin{lemma}\label{lemma_l_r_convergence_tau_i}
  For every $i \in \M$ and $m \geq 1$, we have the following.
  \begin{itemize}
  \item[(i)] If Condition \ref{condition_uniformly_integrable} (i) holds with some $r \geq m$, then we have
  \begin{align}
  {\tau_A^{(i)}} /(-\log A_i) \xrightarrow[A_i \downarrow 0]{\textrm{in $L^m(\P_i)$}} {l(i)}^{-1} \; \textrm{and} \; {D_i^{(c,m)}(\tau_A)} /(-\log A_i) \xrightarrow{A_i \downarrow 0} (c_i/l(i))^m. \label{convergence_l_r_d_a}
  \end{align}
  \item[(ii)] If Condition \ref{condition_uniformly_integrable} (ii) holds with some $r \geq m$, then we have
  \begin{align}
  {\upsilon_B^{(i)}} /(-\log B_{ij(i)}) \xrightarrow[\overline{B}_i \downarrow 0]{\textrm{in $L^m(\P_i)$}} {l(i)}^{-1} \; \textrm{and} \; {D_i^{(c,m)}(\upsilon_B)} /(-\log B_{ij(i)}) \xrightarrow{\overline{B}_i \downarrow 0} (c_i/l(i))^m. \label{convergence_l_r_d_b}
  \end{align}
  \end{itemize}
  \end{lemma}

  Alternatively to Condition \ref{condition_uniformly_integrable}, we
  can use the \emph{$r$-quick convergence}.  The $r$-quick convergence
  of suitable stochastic processes is known to be sufficient for the
  asymptotic optimalities of certain sequential rules based on
  non-i.i.d.\ observations in CPD and SMHT problems and also in the
  diagnosis problem of \cite{Dayanik_2012}.

\begin{definition}[The $r$-quick convergence]
  Let $(\xi_n)_{n\ge 0}$ be any stochastic process and $r > 0$. Then 
\begin{align*}
\rqliminf_{n \rightarrow \infty} \xi_n \geq c
\end{align*}
  if and only if $\E \left[ (T_\delta)^r \right] < \infty$ for every
  $\delta > 0$, where $T_\delta :=\inf \big\{ n \geq 1: \inf_{m \geq n} \xi_m > c - \delta
  \big\}$,  $\delta > 0$.  
\end{definition}




Condition \ref{condition_uniformly_integrable} is implied by  Condition \ref{condition_r_quickly} below; see \cite{Dayanik_2012}, for example.

\begin{condition} \label{condition_r_quickly} For some $r \ge 1$, (i)
  $\rqliminf_{n \uparrow \infty} {\Phi_n^{(i)}}/n \geq l(i)$ under
  $\P_i$, (ii) $\rqliminf_{n \uparrow \infty} {\Psi_n^{(i)}}/n \geq
  l(i)$ under $\P_i$ for every $i\in \M$.
\end{condition}


\begin{proposition}  \label{prop_upper_bound}
  Let $m \geq 1$. (i) If Condition \ref{condition_r_quickly} (i) holds
  for some $r\ge m$, then  Condition
  \ref{condition_uniformly_integrable} (i) and (\ref{convergence_l_r_d_a}) hold.
  (ii) If Condition \ref{condition_r_quickly} (ii) holds for some
  $r\ge m$, then  Condition
  \ref{condition_uniformly_integrable} (ii) and  (\ref{convergence_l_r_d_b}) hold.
  \end{proposition}

\begin{remark} \label{remark_llr_implies_phi} In Condition
  \ref{condition_r_quickly}, (i) implies (ii) by
  (\ref{psi_greater_than_phi}). Moreover, Condition
  \ref{condition_r_quickly} holds if $\rqliminf_{n \uparrow \infty}
   ({\Lambda_n (i,j)} /n) \geq l(i,j)$ under $\P_i$ for every $i\in
  \M$ and $j \in \mathcal{M}_0 \setminus \{i\}$.
\end{remark}

\subsection{Asymptotic Optimality} \label{section_asymptotic_optimality_general}

We now prove the
asymptotic optimalities of $(\tau_A,d_A)$ and $(\upsilon_B,d_B)$ for
Problems \ref{problem_bayes_risk} and \ref{problem_invariant} under
Conditions \ref{condition_uniformly_integrable} (i) and \ref{condition_uniformly_integrable} (ii),
respectively. 

We first derive a lower bound in Lemma  \ref{lemma_lower_bound} below on the expected detection
delay under the optimal strategy.
The lower bound on the expected detection delay under the optimal strategy can be obtained
similarly to CPD and SMHT; see \cite{Baum_Veeravalli1994},
\cite{Dragalin_Tartakovsky_Veeravalli1999},
\cite{Dragalin_Tartakovsky_Veeravalli2000}, \cite{Lai2000},
\cite{Tartakovsky_Veeravalli2004} and \cite{Baron_Tartakovsky2006}.  This lower bound and
Lemma \ref{lemma_l_r_convergence_tau_i} above can be combined to obtain
asymptotic optimality for both problems.


%
%

\begin{lemma} \label{lemma_lower_bound}
For every $i \in \mathcal{M}$ and $j(i)$, we have
\begin{align*}
&\liminf_{\overline{R}_i \downarrow 0} \inf_{(\tau,d) \in \overline{\Delta}
(\overline{R})} \frac {D_i^{(c,m)}(\tau)} {\left({ \frac {c_i} {l(i)} \left|\log \left( \frac 1 {\nu_i}{\sum_{y \in \Y_{j(i)}}\overline{R}_{yi}} \right) \right|} \right)^m} \geq 1.
\end{align*}

\end{lemma}

We now study how to set $A$ in terms of $c$ in order to achieve
asymptotic optimality in Problem \ref{problem_bayes_risk}.
We see from Proposition \ref{corollary_bound_kappa_i_and_r} and Lemma
\ref{lemma_l_r_convergence_tau_i} that the TDL's decrease faster than the EDD and are negligible when $A$ and $B$ are small. Indeed, we have,
in view of the definition of the Bayes risk in (\ref{bayes_risk_i}),
by Proposition \ref{corollary_bound_kappa_i_and_r} and Lemma
\ref{lemma_l_r_convergence_tau_i}, for any $\sigma_i > 0$ for
every $i \in \M$,
\begin{align}
u^{(c,a,m)}_{i} (\tau_{A},d_{A}) \sim  c_i^m \left( \frac {-\log A_i} { l(i)} \right)^m + \sigma_i A_i \sim c_i^m \left( \frac {-\log A_i} { l(i)} \right)^m,  \quad \textrm{as } A_i \downarrow 0. \label{sigma_expression}
\end{align}

This motivates us to choose the value of $A_i$ such that it minimizes
\begin{align}
  g^{(c_i)}_i(x) := c_i^m \left( \frac {-\log x} { l(i)} \right)^m +
  \sigma_i x \label{g_minimizer}
\end{align}
over $x \in (0,\infty)$. Hence, let
\begin{align*}
  A_i(c_i) \in \argmin_{x \in (0,\infty)}g_i^{(c_i)}(x), \quad c_i > 0.
\end{align*}
For example, $A_i(c_i) = c_i /(\sigma_i l(i))$ when $m=1$. It can be easily verified that for every $m \geq 1$ we have $A_i(c_i) \xrightarrow{\|c\| \downarrow 0} 0$ in such a way that $\log A_i(c_i) \sim \log c_i$ as $\|c\| \downarrow 0$. Hence we have
\begin{align*}
u^{(c,a,m)}_{i} (\tau_{A(c)},d_{A(c)}) \sim g_i^{(c_i)} (A_i(c_i))\sim c_i^m \left( \frac {- \log c_i} {l(i)}\right)^m \; \textrm{as } \|c\| \downarrow 0. 
\end{align*}
Consequently, it is sufficient to show that
\begin{align*}
\liminf_{\|c\| \downarrow 0} \frac {\inf_{(\tau,d) \in \Delta}
u^{(c,a,m)}_{i}(\tau,d)}{g_i^{(c_i)} (A_i(c_i))} \geq 1. 
\end{align*}

Its proof is similar to that of Proposition 3.18 of \cite{Dayanik_2012}.

\begin{proposition}[Asymptotic optimality of $(\tau_A,d_A)$ in Problem \ref{problem_bayes_risk}]
Fix $m \geq 1$ and a set of strictly positive constants $a$. Under Condition \ref{condition_uniformly_integrable} (i) or \ref{condition_r_quickly} (i) for the given $m$, the strategy $(\tau_{A(c)},d_{A(c)})$ is asymptotically optimal
as $\|c\| \downarrow 0$; that is, (\ref{optimality_bayes_risk_strong}) holds for every $i \in \M$.
\end{proposition}

It should be remarked here that the asymptotic optimality results hold for any $\sigma_i > 0$.  However, for higher-order approximation, it is ideal to choose its value such that
\begin{align} R_i^{(a)}(\tau_A, d_A) / A_i  \xrightarrow{A_i
    \downarrow 0} \sigma_i. \label{convergence_sigma}
\end{align}
In Section  \ref{section_higher_approximation}, we achieve this value using nonlinear renewal theory.

We now show that the strategy $(\upsilon_B,d_B)$ is asymptotically optimal for Problem \ref{problem_invariant}.
It follows from Proposition \ref{proposition_feasible_strategy} that, if we set
\begin{align*}
B_{ij} (\overline{R}) := \frac {\min_{y \in \Y_j}\overline{R}_{yi}} {\nu_i}, \; \textrm{for every } i \in \M, j \in \M_0 \setminus \{i\},
\end{align*}
then we have $(\upsilon_{B(\overline{R})}, d_{B(\overline{R})}) \in \Delta (\overline{R})$
for every fixed positive constants $\overline{R} = (R_{yi})_{i \in \M, \, y \in \Y \backslash \Y_i }$.  By Lemma \ref{lemma_l_r_convergence_tau_i} (ii),
$\upsilon_{B(\overline{R})} \leq \upsilon^{(i)}_{B(\overline{R})}$ and because
$\min_{y \in \Y_{j(i)}}\overline{R}_{yi} \downarrow 0$ is equivalent to $B_{ij(i)} (\overline{R})
\downarrow 0$,
\begin{align*}
\limsup_{\overline{R}_i \downarrow 0} \frac {D_i^{(c,m)}(\upsilon_{B(\overline{R})})}
{\left( {\frac {c_i} {l(i)} |\log \left( \min_{y \in \Y_{j(i)}}\overline{R}_{yi} / \nu_i \right) |} \right)^m} = \limsup_{\overline{R}_i
\downarrow 0} \frac {D_i^{(c,m)}(\upsilon_{B(\overline{R})})} {\left( { \frac {c_i} {l(i)}|\log
B_{ij(i)} (\overline{R})|}\right)^m}  \leq 1.
\end{align*}
This together with Lemma \ref{lemma_lower_bound} shows the asymptotic optimality.

\begin{proposition}[Asymptotic optimality of $(\upsilon_B, d_B)$ in Problem \ref{problem_invariant}] \label{propositioN_uppwer_bound_fixed_error}
Fix $m \geq 1$. Under Condition \ref{condition_uniformly_integrable} (ii) or \ref{condition_r_quickly} (ii) for the given $m$, the strategy $(\upsilon_{B(\overline{R})},d_{B(\overline{R})})$ is asymptotically optimal
as $\|\overline{R}\| \downarrow 0$; that is, (\ref{optimality_invariant_strong}) holds for every $i \in \M$.
\end{proposition}

\section{Convergence Results of LLR processes} 
\label{section_convergence_results}

In this section, we consider two particular cases where Assumption \ref{condition_a_s_convergence} holds with  $l(i,j)$ expressed in terms of 
the Kullback-Leibler divergence defined below.
We assume that
 $X_\theta, X_{\theta+1},\ldots$ are identically distributed on $\{\mu = i\}$ given $\theta$, for every  $i \in \M$. For the purpose of determining the limit $l(i,j)$, because each class is closed, we can assume without loss of generality that $\mathcal{Y}_{i}$ consists of a single state, say,  
\begin{align}
\mathcal{Y}_{i} = \{i\} \textrm{ with }  f_i(\cdot) \equiv f(i, \cdot), \label{def_y_identical}
\end{align}
for every $i \in \M$.  

The conditional probability of that $Y$ is absorbed by $\mathcal{Y}_i
= \{i\}$ at time $t \geq 0$, given $\{ \mu = i \}$, is
\begin{align}
\rho_t^{(i)} := \P \{ \theta = t | \mu = i\} = \left\{ \begin{array}{ll} \frac {\eta(i)} {\nu_i}, & t=0, \\\frac 1 {\nu_i}\sum_{(y_0,\ldots, y_{t-1}) \in \mathcal{Y}_{0}^t} \eta(y_0) \prod_{k=1}^{t-1} P(y_{k-1}, y_k) P(y_{t-1},i), & t\geq 1. \end{array} \right. \label{def_rho_t_i}
\end{align}
We assume the following throughout this section.
\begin{assumption} \label{assumption_rho}
For every $i \in \M$, we assume that
\begin{align}
\varrho^{(i)} := - \lim_{t \rightarrow \infty}\frac {\log \rho_t^{(i)}} t = - \lim_{t \rightarrow \infty}\frac {\log (1-\sum_{k=0}^t \rho_k^{(i)})} t \in (0, \infty] \label{exponential_tail}
\end{align}
exists. 
\end{assumption}
Here, $\varrho^{(i)} = \infty$ holds for example when $\P_i \left\{ \theta < M \right\} = 1$ for some $M <\infty$.  In a special case where the change time is geometric with parameter $p > 0$ as in \cite{Dayanik_2012}, this is satisfied with $\varrho^{(i)} = |\log (1-p)|$. Assumption \ref{assumption_rho} also holds, for example, when $\theta$ is a mixture or a sum of geometric random variables; see the examples given in Section \ref{numerics_convergence}.

\subsection{Example 1} \label{subsection_example1} Suppose that the
distribution of $X$ is identical also in the transient set
$\mathcal{Y}_{0}$; namely,
\begin{align*}
f(y,\cdot) = f(z, \cdot) =: f_0(\cdot), \quad y,z \in \mathcal{Y}_{0}.
\end{align*}
%
We denote the Kullback-Leibler divergence
of $f_i(\cdot)$ from $f_j(\cdot)$ by
\begin{align}
  q(i,j) := \int_E \left( \log \frac {f_i(x)}{f_j(x)}\right) f_i(x)
  m(\diff x), \quad  i\in\M,\;j \in \M_0 \setminus \{i\}, \label{def_q_i_j}
\end{align}
which always exists and is nonnegative.

We assume $f_i(\cdot)$ and $f_j(\cdot)$ as in \eqref{def_y_identical}
for any $i \neq j$ are distinguishable; namely, we assume the
following.
\begin{assumption} \label{assumption_finiteness_ratio} We assume $\int_{\{x\in E: f_i(x)
      \not= f_j(x)\}} f_i(x) m (\diff x)> 0$ for every $i
  \in \M$ and $j \in \M_0 \setminus \{i\}$.  This ensures that \begin{align}
  q(i,j) > 0, \quad i\in\M,\; j \in \M_0 \setminus \{i\}.
  \label{positiveness_kullback_leibler}
\end{align}
\end{assumption}

To ensure that $\int_E \big( \log \frac {f_0(x)} {f_j(x)} \big) f_i(x) m(\diff x)$ exists
for every $i\in\M$ and $j \in \M_0 \setminus \{i\}$, we further assume the following.

\begin{assumption} \label{assumption_q_i_0} For every $i \in \M$, we
  assume that $q(i,0) < \infty$.
\end{assumption}
Indeed, since $\int_E (\log
\frac {f_i(x)} {f_j(x)})_- f_i(x) m (\diff x)\le 1$ for every $i\in \M$ and $j\in
\M_0\setminus\{i\}$, 
\begin{align}
  \label{definition_q_i_j_0}
  \int_E \left( \log \frac {f_0(x)}{f_j(x)} \right) f_i(x) m (\diff x)=
  \int_E \left( \log \frac {f_i(x)}{f_j(x)} \right) f_i(x) m (\diff x) -
  \int_E \left( \log \frac {f_i(x)}{f_0(x)} \right) f_i(x) m (\diff x)
= q(i,j)-q(i,0),
\end{align}
exists by
Assumption \ref{assumption_q_i_0}.  Finally, we assume the following.

\begin{assumption} \label{condition_rho_q_finite}
For every $i \in \M$ and $j \in \M \setminus \{i\}$, we assume $\min \{ \varrho^{(j)},  q(i,j) \} < \infty$. 
\end{assumption}

We shall prove the following under Assumptions \ref{assumption_rho}-\ref{condition_rho_q_finite}.
\begin{proposition}[Limits of LLR processes in Example 1] \label{prop_convergence_example_1}
For every $i \in \M$, Assumption \ref{condition_a_s_convergence} holds with the limits
\begin{align}    
  l(i,j) &:= \left\{ 
    \begin{aligned}
      &q(i,0)+  \min_{k \in \M}  \varrho^{(k)}, && j=0 \\
      &\min\big\{q(i,j),q(i,0) + \varrho^{(j)} \big\}, && j \in
      \M\setminus\{i\}
    \end{aligned}
  \right\}  
\equiv \left\{
    \begin{aligned}
&q(i,0)+  \min_{k \in \M}  \varrho^{(k)}, && j=0 \\
   &q(i,j), &&  j \in \Gamma_i \\
      &q(i,0) + \varrho^{(j)}, && j \in \M \setminus (\Gamma_i \cup \{i\})
    \end{aligned} \right\}\label{definition_limit_llr_2},
\end{align}
where $\Gamma_i := \{ j \in \M \setminus \{i\}: q(i,j) < q(i,0) +
    \varrho^{(j)} \}$. 
\end{proposition}

\begin{remark} \label{remark_regarding_l} 
\begin{enumerate}
\item  Assumptions  \ref{assumption_q_i_0} and \ref{condition_rho_q_finite} ensure that
\begin{align*}
q(i,j) &< \infty, \quad j \in \Gamma_i, \\
q(i,0) + \varrho^{(j)} &< \infty, \quad j \in \M \setminus (\Gamma_i \cup \{i\}).
\end{align*}
\item Assumption \ref{assumption_finiteness_ratio} guarantees that
  $l(i,j) > 0$ for every $i \in \M$ and $j \in \M_0 \setminus \{i\}$.  In particular, (1) ensures $0 <  l(i,j) < \infty$ for any $j \in \M \setminus \{i\}$. Hence, $0 < l(i) < \infty$. 
\item By \eqref{definition_limit_llr_2}, we can choose $j(i) \in \{0\} \cup \Gamma_i$.  If $j(i) = 0$, we must have $\min_{k \in \M} \varrho^{(k)} < \infty$.
\item We assume in Section \ref{section_higher_approximation} for
  higher-order approximations that,  for every $i\in \M$, there is a unique $j(i)\in
  \M_0\setminus\{i\}$
  such that $l(i) = l(i,j(i)) = \min_{j \in \M_0 \setminus\{i\}}
  l(i,j)$.  Contrary to the case $\theta$ is
  geometric as in \cite{Dayanik_2012}, the uniqueness of $j(i)$ does
  not exclude the case $j(i) = 0$.  In particular, for the case $j(i)
  = 0$, the uniqueness implies that $\varrho^{(k)}$ is uniquely
  minimized by $k=i$.  On the other hand, if $j(i) \in \Gamma_i$, then
  $l(i) < l(i,0)$, $q(i,j(i)) <\min_{j \in \M} (q(i,0) +
  \varrho^{(j)})$, and $\Gamma_i\not= \varnothing$.

\end{enumerate}
\end{remark}

In order to show Proposition \ref{prop_convergence_example_1}, we first simplify the LLR process as  in \eqref{eq:log-likelihood-ratio-processes}. Define, for each $j \in \M$,
\begin{align}\label{definition_K}
\begin{split}
L_n^{(j)} &:= \log \left( \rho_0^{(j)} + \sum_{k=1}^n
   \rho_k^{(j)}  \prod_{l=1}^{k-1} \frac {f_0(X_l)} {f_j(X_l)} \right),  \\
K_n^{(j)} &:= \log \left( \frac {\rho_0^{(j)}} {\rho_n^{(j)}} \prod_{k=1}^n \frac {f_j(X_k)} {f_0(X_k)} + \sum_{k=1}^n
    \frac {\rho_k^{(j)}} {\rho_n^{(j)}} \prod_{m=k}^n \frac {f_j (X_m)} {f_0(X_m)} \right) = -\log \rho_n^{(j)} + \sum_{k=1}^n \log \frac {f_j(X_k)} {f_0(X_k)}  + L_n^{(j)}. 
\end{split}
\end{align}

\begin{lemma} \label{lemma_example_1_lambda} Fix $i \in \M$.
For any $n \geq 1$,
\begin{align*}
\Lambda_n(i,0) &=  \sum_{k=1}^n \log \frac { f_i(X_k)} { f_0(X_k)}  + L_n^{(i)} -  \log \Big[ \sum_{j \in \M} \nu_j \big(1- \sum_{t=0}^n \rho_t^{(j)} \big)  \Big] + \log \nu_i 
\end{align*}
and for  $j \in \M \backslash \{i\}$
\begin{align*}
\Lambda_n(i,j) &= \sum_{k=1}^n \log \frac { f_i(X_k)} { f_j(X_k)}  + L_n^{(i)} - L_n^{(j)} +  \log \nu_i - \log \nu_j  \\ &= -\log \rho_n^{(j)} + \sum_{k=1}^n \log \frac { f_i(X_k)} { f_0(X_k)}  + L_n^{(i)} - K_n^{(j)} +\log \nu_i - \log \nu_j .
\end{align*}

\end{lemma}

By this lemma, each LLR process admits a decomposition
\begin{align*}
  \Lambda_n(i,j)=  \sum_{l=1}^n h_{ij}(X_l) + \epsilon_n(i,j), \quad j\in \M_0\setminus \{i\},
\end{align*}
where
\begin{align}
   h_{ij}(x) &:= \left\{
    \begin{aligned}
&\log \frac {f_i(x)}{f_0(x)} + \min_{k \in \M}  \varrho^{(k)}, && j = 0\\
      &\log \frac {f_i(x)}{f_j(x)}, && j \in \Gamma_i \\
&\log \frac {f_i(x)}{f_0(x)} + \varrho^{(j)}, && j \in \M \setminus (\Gamma_i \cup\{i\})
    \end{aligned} \right\}, \quad x \in E, \nonumber \\ \label{definition_epsilon}
  \epsilon_n(i,j) &:= \left\{
    \begin{aligned}
      &L^{(i)}_n  -   \log \Big[ \sum_{j \in \M} \nu_j \big(1- \sum_{t=0}^n \rho_t^{(j)} \big)  \Big] - n \min_{k \in \M}  \varrho^{(k)}+ \log \nu_i, && j=0 \\
      &L_n^{(i)} - L_n^{(j)}+\log \nu_i - \log \nu_j, && j \in \Gamma_i \\
      &L_n^{(i)} - K_n^{(j)}+ \log \nu_i - \log \nu_j - \log \rho^{(j)}_n - n \varrho^{(j)}, && j \in
      \M\setminus (\Gamma_i \cup \{i\})
    \end{aligned} \right\}, \quad n\ge 1.
\end{align}
Here notice that $\varrho^{(j)} < \infty$ for $j \in \M \setminus (\Gamma_i \cup \{i\})$ by Assumption \ref{condition_rho_q_finite}.


We explore the convergence  for $ {(\sum_{l=1}^n h_{ij}(X_l))}/n$ and $\epsilon_n(i,j)/n$ separately.  For $i \in \M$ and $j
  \in \M_0 \setminus \{i\}$, because $\theta$ is an a.s.\ finite random variable, a direct application of the strong law of large number (SLLN) leads to
\begin{align}
\frac 1 n {\sum_{l=1}^n h_{ij}(X_l)}
  \xrightarrow[n \uparrow \infty]{\text{$\P_i$-a.s.}} l(i,j).\label{convergence_h}
\end{align}

We now show that $\epsilon_n(i,j)$ in \eqref{definition_epsilon}
converges almost surely to
zero.  

%

\begin{lemma} \label{lemma_limit_l} 

For every $i \in \M$, we have the followings under $\P_i$.
  \begin{enumerate}
  \item[(i)] $L_n^{(i)}/n \xrightarrow{n \uparrow \infty} 0$ a.s.
  \item[(ii)]  $L_n^{(j)}/n \xrightarrow{n \uparrow \infty}
    \left[q(i,j)-q(i,0)-\varrho^{(j)} \right]_+$ a.s.\ for every $j \in
    \M\setminus \{i\}$.
  \item[(iii)]  $K_n^{(j)}/n \xrightarrow{n \uparrow \infty}
    \left[q(i,j)-q(i,0)-\varrho^{(j)} \right]_-$ a.s.\ for every $j \in
    \M\setminus (\Gamma_i \cup\{i\})$.
  \item[(iv)]  $L_n^{(i)}$ converges a.s.\ as $n \uparrow \infty$ to an a.s.
    finite random variable $L_\infty^{(i)}$.
  \item[(v)]  $L_n^{(j)}$ converges a.s.\ as $n \uparrow \infty$ to an a.s.
    finite random variable $L_\infty^{(j)}$ for every $j \in
    \Gamma_i$.
  \item[(vi)]  For every $j\in \M$, $(|L^{(j)}_n/n|^r)_{n\ge 1}$ is
    uniformly integrable for every $r\ge 1$, if
    \begin{align}
       \int_{E}\frac {f_0(x)} {f_j(x)} f_0(x) m (\diff x)  <\infty \quad \text{and}
      \quad \int_{E}\frac {f_0(x)} {f_j(x)} f_i(x) m (\diff x)  <\infty. \label{eq:sufficient-for-UI-of-L-average}
    \end{align}
  \item[(vii)]  For every $j\in \M \setminus (\Gamma_i \cup \{i\})$, $(|K^{(j)}_n/n|^q)_{n\ge 1}$ is
    uniformly integrable for every $0\le q \le r$, if
    (vi) holds, and
    \begin{align}
      \int_{E}\left| \frac {f_j(x)} {f_0(x)}\right|^r f_0(x) m (\diff x)  <\infty \quad \text{and} \quad  \int_{E}\left| \frac {f_j(x)} {f_0(x)}\right|^r f_i(x) m (\diff x)  <\infty, \quad \textrm{for some } r \geq 1. \label{eq:sufficient-for-UI-of-K-average}
    \end{align}
  \end{enumerate}
\end{lemma}


By the characterization of $\epsilon_n(i,j)$ in
(\ref{definition_epsilon}) and Lemma \ref{lemma_limit_l} (i)-(iii),
\begin{align*}
 {\epsilon_n(i,j)} /n \xrightarrow{\P_i-a.s.} 0, \quad i \in \M, j\in \M\setminus\{i\}.
\end{align*}
This also holds when $j = 0$ because 
\begin{align}-  \frac 1 n \log \Big[ \sum_{j \in \M} \nu_j \big(1- \sum_{t=0}^n \rho_t^{(j)}\big) \Big]   \xrightarrow{n \uparrow \infty} \min_{j \in \mathcal{M}} \varrho^{(j)}. \label{rho_convergence}
\end{align}
Indeed, the left-hand side of \eqref{rho_convergence} equals
\begin{align*}
  - \frac 1 n \log \big[ \sum_{j \in \M} \exp \big( \log \nu_j + \log
  \big(1- \sum_{t=0}^n \rho_t^{(j)} \big) \big) \big] = - \frac 1 n
  \log \big( \sum_{j \in \M}e^{-n A_j(n)} \big),
\end{align*}
where $A_j(n) := - \frac 1 n \big( \log \nu_j + \log \big(1- \sum_{t=0}^n \rho_t^{(j)} \big) \big)$.
Because $A_j(n) \rightarrow \varrho^{(j)}$ by Assumption \ref{assumption_rho} and by Lemma \ref{lemma_convergence_min_to_mu}, we have \eqref{rho_convergence}.
This together with \eqref{convergence_h} shows Proposition \ref{prop_convergence_example_1}.  

The a.s.\ convergence can be extended to the $L^r (\P_i)$-convergence
for $r \geq 1$ as well, under additional integrability conditions.
Firstly, as in Lemma 4.3 of \cite{Dayanik_2012}, for every $i \in \M$,
$j \in \M_0 \setminus \{i\}$ and $r \geq 1$, we have $(1/n)
{\sum_{l=1}^n h_{ij}(X_l)} \xrightarrow[n \uparrow \infty]{L^r(\P_i)}
l(i,j)$, if
\begin{align}
  \int_E \left| h_{ij}(x) \right|^r f_0(x) m(\diff x)<\infty \quad
  \text{and} \quad \int_E \left| h_{ij}(x) \right|^r f_i(x) m(\diff x)<\infty. \label{condition_convergence_h}
\end{align}
Here, (\ref{condition_convergence_h}) holds if the following condition
holds.

\begin{condition} \label{condition_convergence_h_l_r} Given $i \in
  \M$, $j\in \M_0\setminus\{i\}$, and $r \geq 1$, suppose that
  \begin{align*} 
    \begin{aligned}
      &\int_E \left| \log \frac{f_i(x)}{f_j(x)}\right|^r f_0(x) m(\diff x)
      <\infty \quad \text{and} \quad \int_E \left| \log \frac{f_i(x)}{f_j(x)}\right|^r f_i(x) m(\diff x)<\infty && \text{if}\quad
      j\in \Gamma_i,\\
      &\int_E \left| \log \frac{f_i(x)}{f_0(x)}\right|^r f_0(x) m(\diff x)
      <\infty \quad \text{and} \quad \int_E \left| \log \frac{f_i(x)}{f_0(x)}\right|^r f_i(x) m(\diff x) <\infty && \text{if} \quad
      j \in \M_0 \backslash \Gamma_i.
    \end{aligned}
  \end{align*}  
In addition, when $j = 0$,  we assume $\min_{j \in \M}\varrho^{(j)} < \infty$.
\end{condition}

On the other hand, by Lemma \ref{lemma_limit_l}, $\epsilon_n(i,j)/n
  \rightarrow 0$ as $n \uparrow \infty$ in $L^r(\P_i)$ under Condition \ref{condition_convergence_l_r_epsilon} below.  Notice in Lemma \ref{lemma_limit_l} (vi) that in order for $L_n^{(i)}$ to converge in $L^r(\P_i)$ to zero, it is sufficient to have
\begin{align}
 \int_E \frac {f_0 (x)} {f_i(x)} f_0(x) m (\diff x) < \infty, \label{condition_l_r_l_i}
\end{align}
because $\int_E \frac {f_0 (x)} {f_i(x)} f_i(x) m (\diff x) = \int_E f_0 (x)  m (\diff x)= 1 < \infty$.  

\begin{condition} \label{condition_convergence_l_r_epsilon} Given $i
  \in \M$, $j\in \M \setminus\{i\}$ and $r \geq 1$, we suppose that
 (\ref{eq:sufficient-for-UI-of-L-average}) and   (\ref{condition_l_r_l_i}) hold, and, if $j \in \M \setminus \Gamma_i$,  (\ref{eq:sufficient-for-UI-of-K-average})   holds for the given $r$.
\end{condition}



In summary, we have the following $L^r$-convergence results.

\begin{proposition} \label{proposition_convergence_llr} For every $i
  \in \M$ and $j\in \M_0\setminus\{i\}$, we have $\Lambda_n(i,j)/n \rightarrow l(i,j)$ as $n \uparrow \infty$ in $L^r(\P_i)$ for some $r \geq 1$ if
Conditions \ref{condition_convergence_h_l_r} and
\ref{condition_convergence_l_r_epsilon} hold for the given $r$.


\end{proposition}

\subsection{Example 2} \label{subsection_example_2}

As a variant of Example 1, we consider the case $X$ is not necessarily identically distributed in $\Y_0$. 
Suppose $\mathcal{Y}_0 = \mathcal{Y}^{(1)}_{0} \sqcup \cdots \sqcup \mathcal{Y}^{(M)}_{0}$ and $\mathcal{Y}_{0}^{(i)}$ is absorbed with probability one by $\mathcal{Y}_i = \{ i\}$ for each $i \in \M$.  This implies that
\begin{align*}
\P \big\{ \mu = i |  Y_0 \in \Y_0^{(i)}\big\}=1, \quad i \in \M.
\end{align*} 
Also let 
\begin{align*}
f(y,\cdot) \equiv f(z, \cdot) =: f^{(0)}_i(\cdot), \quad y,z \in \mathcal{Y}_{0}^{(i)}, i \in \M.
\end{align*}
The conditional probability of $\theta = t$ given $\{ \mu = i\}$ as in \eqref{def_rho_t_i} can be written
\begin{align*}
\rho_t^{(i)} = \left\{ \begin{array}{ll} \frac {\eta(i)} {\nu_i}, & t=0, \\\frac 1 {\nu_i}\sum_{y_0,\ldots, y_{t-1} \in \mathcal{Y}_{0}^{(i)}} \eta(y_0) \prod_{k=1}^{t-1} P(y_{k-1}, y_k) P(y_{t-1},i), & t\geq 1. \end{array} \right.
\end{align*}

\begin{assumption} \label{assumption_finiteness_ratio_2}  For every $i
  \in \M$, we assume $f_i(\cdot)$ is distinguishable from $f_j(\cdot)$  for  $j \in \M \setminus \{i\}$ and from  $f^{(0)}_j(\cdot)$ for every $i \in \M$; $\int_{\{x\in E: f_i(x)
      \not= f_j(x)\}} f_i(x) m (\diff x)> 0$ and $\int_{\{x\in E: f_i(x)
      \not= f^{(0)}_j(x)\}} f_i(x) m (\diff x)> 0$. This ensures that $q(i,j) > 0$ and $q^{(0)}(i, j) > 0$ where we use \eqref{def_q_i_j} and  define
\begin{align*}
  q^{(0)}(i,j) := \int_E \left( \log \frac {f_i(x)}{f_j^{(0)}(x)}\right) f_i(x)
  m(\diff x), \quad  i, j \in \M.
\end{align*}
\end{assumption}

We assume the following to ensure that $\int_E \big( \log \frac {f^{(0)}_j(x)} {f_j(x)} \big) f_i(x) m(\diff x)$ exists
for every $i, j \in \M$.

\begin{assumption} \label{assumption_q_i_0_2} For every $i, j\in \M$, we
  assume that $q^{(0)}(i,j) < \infty$.
\end{assumption}
We shall show the following under Assumptions \ref{assumption_rho}, \ref{condition_rho_q_finite}, \ref{assumption_finiteness_ratio_2}, and \ref{assumption_q_i_0_2}.
\begin{proposition}[Limits of LLR processes in Example 2] \label{prop_convergence_example_2}
Assumption \ref{condition_a_s_convergence} holds with the limits
\begin{align}    
  l(i,j) &:= \left\{ 
    \begin{aligned}
      &\min_{k \in \M} \big\{ q^{(0)}(i,k)+\varrho^{(k)} \big\}, && j=0 \\
      &\min\big\{q(i,j),q^{(0)}(i,j) + \varrho^{(j)} \big\}, && j \in
      \M\setminus\{i\}
    \end{aligned}
  \right\}  \equiv \left\{
    \begin{aligned}
& \min_{k \in \M} \big\{ q^{(0)}(i,k)+\varrho^{(k)} \big\}, && j=0 \\
      &q(i,j), &&  j \in \Gamma_i \\
      &q^{(0)}(i,j) + \varrho^{(j)}, && j \in \M \setminus (\Gamma_i \cup \{i\})
    \end{aligned} \right\}, \label{definition_limit_llr_3}
\end{align}
where $\Gamma_i := \big\{ j \in \M \setminus \{i\}: q(i,j) < q^{(0)}(i,j) +
    \varrho^{(j)} \big\}$ for every $i\in \M$. 
\end{proposition}


\begin{remark} \label{remark_example_2}
\begin{enumerate}
\item  Assumptions \ref{condition_rho_q_finite} and \ref{assumption_q_i_0_2}  ensure that
\begin{align*}
q(i,j) &< \infty, \quad j \in \Gamma_i, \\
q^{(0)}(i,j) + \varrho^{(j)} &< \infty, \quad j \in \M \setminus (\Gamma_i \cup \{i\}).
\end{align*}
\item Assumption \ref{assumption_finiteness_ratio_2}  guarantees that
  $l(i,j) > 0$ for every $i \in \M$ and $j \in \M_0 \setminus \{i\}$.  In particular, by (1) $0 <  l(i,j) < \infty$ for any $j \in \M \setminus \{i\}$. Hence, $0 < l(i) < \infty$. 
\item
By \eqref{definition_limit_llr_3}, we can choose $j(i) \in \{0\} \cup \Gamma_i$.  If $j(i) = 0$, we must have $\min_{k \in \M} \varrho^{(k)} < \infty$.
\item Suppose there is a unique $j(i)\in
  \M_0\setminus\{i\}$
  such that $l(i) = l(i,j(i)) = \min_{j \in \M_0 \setminus\{i\}}
  l(i,j)$ for every $i\in \M$.  As in Example 1, $j(i)$ can be $0$ (or
  in $\Gamma_i$).  In particular, for the case $j(i) = 0$, then the
  uniqueness implies that $\argmin\big\{q^{(0)}(i,j) + \varrho^{(j)}
  \big\}$ is uniquely attained by $j=i$.  On the other hand, if $j(i)
  \in \Gamma_i$, then $l(i) < l(i,0)$, $q(i,j(i)) <\min_{j \in \M}
  (q^{(0)}(i,j) + \varrho^{(j)})$, and $\Gamma_i\not= \varnothing$.
\end{enumerate}
\end{remark}

%


Similarly to Example 1 of Section  \ref{subsection_example1}, we simplify the LLR process as follows.  Define
\begin{align*}
\Lambda_n^{(0)}(i,j) &:= \log \frac {\widetilde{\Pi}_n^{(i)}}{\sum_{y \in \mathcal{Y}_0^{(j)}}\Pi_n(y)}, \quad i,j \in \M;
\end{align*}
we later show that $\Lambda_n(i,0)/n \sim \min_{j \in \M} \Lambda_n^{(0)}(i,j)/n$ as $n \rightarrow \infty$ under $\P_i$ (see \eqref{lambda_convergence_min} below).
\begin{lemma} \label{lemma_example_2_lambda}
For $i,j \in \M$, we have
\begin{align*}
\Lambda_n^{(0)}(i,j) 
&=  \sum_{k=1}^n \log \frac { f_i(X_k)} { f^{(0)}_j(X_k)}  + L_n^{(i)} -  \log \big( 1- \sum_{t=0}^n \rho_t^{(j)}  \big) + \log \nu_i - \log \nu_j,
\end{align*}
and for $i \in \M$ and $j \in \M \setminus\{i\}$
\begin{align*}
\Lambda_n(i,j) 
&=\sum_{k=1}^n \log \frac { f_i(X_k)} { f_j(X_k)}  + L_n^{(i)} - L_n^{(j)} + \log \nu_i - \log \nu_j\\ &= - \log \rho_n^{(j)} + \sum_{k=1}^n \log \frac { f_i(X_k)} { f^{(0)}_j(X_k)}  + L_n^{(i)} - K_n^{(j)} + \log \nu_i - \log \nu_j,
\end{align*}
where for each $j \in \M$
\begin{align*}
L_n^{(j)} &:= \log \left( \rho_0^{(j)} + \sum_{k=1}^n
   \rho_k^{(j)}  \prod_{l=1}^{k-1} \frac {f^{(0)}_j(X_l)} {f_j(X_l)} \right), \\
K_n^{(j)} &:= \log \left( \frac {\rho_0^{(j)}} {\rho_n^{(j)}} \prod_{k=1}^n \frac {f_j(X_k)} {f^{(0)}_j(X_k)} + \sum_{k=1}^n
    \frac {\rho_k^{(j)}} {\rho_n^{(j)}} \prod_{m=k}^n \frac {f_j (X_m)} {f^{(0)}_j(X_m)} \right) = -\log \rho_n^{(j)} + \sum_{k=1}^n \log \frac {f_j(X_k)} {f^{(0)}_j(X_k)}  + L_n^{(j)}.
\end{align*}
\end{lemma}


As in Example 1, we decompose each LLR process for every $i\in \M$ such that
\begin{align*}
  \Lambda_n(i,j) &=  \sum_{l=1}^n h_{ij}(X_l) + \epsilon_n(i,j), \quad j \in \M \setminus \{i\},\\
  \Lambda_n^{(0)}(i,j) &=  \sum_{l=1}^n h^{(0)}_{ij}(X_l) + \epsilon_n^{(0)}(i,j), \quad j \in \M,
\end{align*}
where
\begin{align*}
   h_{ij}(x) &:= \left\{
    \begin{aligned}
      &\log \frac {f_i(x)}{f_j(x)}, && j \in \Gamma_i \\
&\log \frac {f_i(x)}{f^{(0)}_j(x)} + \varrho^{(j)}, && j \in \M \setminus (\Gamma_i \cup\{i\})
    \end{aligned} \right\}, \quad x \in E, \\ 
h^{(0)}_{ij}(x) &:= \log \frac {f_i(x)}{f^{(0)}_j(x)} + \varrho^{(j)}, \quad x \in E, \\
  \epsilon_n(i,j) &:= \left\{
    \begin{aligned}
      &L_n^{(i)} - L_n^{(j)}+\log \nu_i - \log \nu_j, && j \in \Gamma_i \\
      &L_n^{(i)} - K_n^{(j)}+ \log \nu_i - \log \nu_j - \log \rho^{(j)}_n - n \varrho^{(j)}, && j \in
      \M\setminus (\Gamma_i \cup \{i\})
    \end{aligned} \right\}, \quad n\ge 1, \\
\epsilon_n^{(0)}(i,j) &:= L_n^{(i)}  -   {\log \big( 1- \sum_{t=0}^n \rho_t^{(j)}  \big)} - n \varrho^{(j)} +\log \nu_i - \log \nu_j, \quad n \geq 1.
\end{align*}
By the SLLN and Assumption \ref{assumption_rho}, for every $i \in \M$, we have  $\P_i$-a.s.\ as $n \uparrow \infty$
\begin{align}\label{convergence_h_2}
\begin{split}
\frac 1 n {\sum_{l=1}^n h_{ij}(X_l)}
  &\longrightarrow l(i,j), \quad j \in \M \setminus \{i\}\\
 \frac 1 n {\sum_{l=1}^n h_{ij}^{(0)}(X_l)}
  &\longrightarrow q^{(0)}(i,j)+\varrho^{(j)}, \quad j \in \M.
\end{split}
\end{align}

We now show that $\epsilon_n(i,j)$ converges almost surely to zero as
$n\to \infty$.  Similarly to Lemma \ref{lemma_limit_l}, the following
holds.

\begin{lemma} \label{lemma_limit_2} 
For every $i \in \M$, we have the followings under $\P_i$.
  \begin{enumerate}
  \item[(i)] $L_n^{(i)}/n \xrightarrow{n \uparrow \infty} 0$ a.s.
  \item[(ii)] $L_n^{(j)}/n \xrightarrow{n \uparrow \infty}
    \left[q(i,j)-q^{(0)}(i,j)-\varrho^{(j)} \right]_+$ a.s.\ for every $j \in
    \M\setminus \{i\}$.
  \item[(iii)] $K_n^{(j)}/n \xrightarrow{n \uparrow \infty}
    \left[q(i,j)-q^{(0)}(i,j)-\varrho^{(j)} \right]_-$ a.s.\ for every $j \in
    \M\setminus (\Gamma_i \cup \{i\})$.
  \item[(iv)] $L_n^{(i)}$ converges a.s.\ as $n \uparrow \infty$ to an a.s.
    finite random variable $L_\infty^{(i)}$.
  \item[(v)] $L_n^{(j)}$ converges a.s.\ as $n \uparrow \infty$ to an a.s.
    finite random variable $L_\infty^{(j)}$ for every $j \in
    \Gamma_i$.
  \item[(vi)] For every $j\in \M$, $(|L^{(j)}_n/n|^r)_{n\ge 1}$ is
    uniformly integrable for every $r\ge 1$, if
    \begin{align}
       \int_{E}\frac {f^{(0)}_j(x)} {f_j(x)} f^{(0)}_i(x) m (\diff x)  <\infty \quad \text{and}
      \quad \int_{E}\frac {f^{(0)}_j(x)} {f_j(x)} f_i(x) m (\diff x)  <\infty. \label{eq:sufficient-for-UI-of-L-average_2}
    \end{align}
  \item[(vii)] For every $j\in \M \setminus (\Gamma_i \cup \{i\})$, $(|K^{(j)}_n/n|^q)_{n\ge 1}$ is
    uniformly integrable for every $0\le q \le r$, if
    (\ref{eq:sufficient-for-UI-of-L-average_2}) holds and
    \begin{align}
      \int_{E}\left| \frac {f_j(x)} {f^{(0)}_j(x)}\right|^r f^{(0)}_i(x) m (\diff x)  <\infty \quad \text{and} \quad  \int_{E}\left| \frac {f_j(x)} {f_j^{(0)}(x)}\right|^r f_i(x) m (\diff x)  <\infty, \quad \textrm{for some } r \geq 1. \label{eq:sufficient-for-UI-of-K-average_2}
    \end{align}
  \end{enumerate}
\end{lemma}
By this lemma, for every $i \in \M$, we have $\epsilon_n(i,j)/n \rightarrow 0$ for $j\in \M \setminus\{i\}$, and $\epsilon_n^{(0)}(i,j)/n \rightarrow 0$ for $j\in \M$, as $n
  \uparrow \infty$ $\P_i$-a.s.
This together with Lemma \ref{convergence_h_2} shows Proposition \ref{prop_convergence_example_2}, once we show that
\begin{align}
\frac 1 n \Lambda_n(i,0) \xrightarrow[n \uparrow \infty]{\textrm{$\P_i$-a.s.}} \min_{j \in \M} \big\{ q^{(0)}(i,j)+\varrho^{(j)} \big\}. \label{lambda_convergence_min}
\end{align}
Indeed, 
\begin{align}
\frac 1 n \Lambda_n(i,0) = \frac 1 n  \log \left( \frac {\widetilde{\Pi}_n^{(i)}}{\sum_{j \in \M} \sum_{y \in \Y_j^{(0)}}\Pi_n(y)} \right) = -\frac 1 n  \log \left( \sum_{j \in \M} \frac {\sum_{y \in \Y_j^{(0)}}\Pi_n(y) } {\widetilde{\Pi}_n^{(i)}} \right)  
= -\frac 1 n  \log \left( \sum_{j \in \M} e^{-n A_n^{(j)}} \right)  \label{lambda_zero_relationship}
\end{align}
where $A_n^{(j)}:= \Lambda_n^{(0)}(i,j)/n \rightarrow q^{(0)}(i,j)+\varrho^{(j)}$ as $n \uparrow \infty$ $\P_i$-a.s. Hence by Lemma \ref{lemma_convergence_min_to_mu},  \eqref{lambda_convergence_min} holds.

We now pursue the convergence in the $L^r$-sense.  
In view of \eqref{lambda_zero_relationship}, we have $\Lambda_n(i,0)/n \leq \Lambda_n^{(0)}(i,j)/n$ for any $j \in \M$
and
\begin{align*}
\frac 1 n \Lambda_n(i,0)   \geq -\frac 1 n  \log \left( M \max_{j \in \M} \frac {\sum_{y \in \Y_j^{(0)}}\Pi_n(y) } {\widetilde{\Pi}_n^{(i)}} \right)  
= - \frac {\log M} n + \min_{j \in \M}  \frac 1 n  \Lambda_n^{(0)}(i,j) \geq - \frac {\log M} n - \sum_{j \in \M}  \frac 1 n  (\Lambda_n^{(0)}(i,j))_-.
\end{align*}
Therefore, for the proof of the uniform integrability of $\Lambda_n(i,0)/n$, it is sufficient to show that of $\Lambda_n^{(0)}(i,j)/n$ for every $j \in \M$.

As in Example 1, for every $i \in \M$ and $r \geq 1$,  we have $(1/n)
  {\sum_{l=1}^n h_{ij}(X_l)} \xrightarrow[n \uparrow
  \infty]{L^r(\P_i)} l(i,j)$ for $j \in \M \setminus \{i\}$, if
\begin{align*}
  \int_E \left| h_{ij}(x) \right|^r f^{(0)}_i(x) m(\diff x)<\infty \quad
  \text{and} \quad \int_E \left| h_{ij}(x) \right|^r f_i(x) m(\diff x)<\infty, 
\end{align*}
which are satisfied under Condition \ref{condition_convergence_h_l_r_2} below.
\begin{condition} \label{condition_convergence_h_l_r_2} For given $i \in
  \M$, $j\in \M\setminus\{i\}$, and $r \geq 1$, suppose that if $ j\in \Gamma_i$
  \begin{align*} 
      \int_E \left| \log \frac{f_i(x)}{f_j(x)}\right|^r f^{(0)}_i(x) m(\diff x)
      <\infty \quad \text{and} \quad \int_E \left| \log \frac{f_i(x)}{f_j(x)}\right|^r f_i(x) m(\diff x)<\infty, 
\end{align*}
and if $j\in \M \setminus (\Gamma_i \cup \{i\})$
\begin{align}
      \int_E \left| \log \frac{f_i(x)}{f_j^{(0)}(x)}\right|^r f^{(0)}_i(x) m(\diff x)
      <\infty \quad \text{and} \quad \int_E \left| \log \frac{f_i(x)}{f^{(0)}_j(x)}\right|^r f_i(x) m(\diff x) <\infty.  \label{cond_f_i_j_0}
  \end{align}  
\end{condition}
Moreover, $(1/n)
  {\sum_{l=1}^n h_{ij}^{(0)}(X_l)} \xrightarrow[n \uparrow
  \infty]{L^r(\P_i)} q^{(0)}(i,j)+\varrho^{(j)}$ for $j \in \M$, if 
\begin{align*}
  \int_E \left| h_{ij}^{(0)}(x) \right|^r f^{(0)}_i(x) m(\diff x)<\infty \quad
  \text{and} \quad \int_E \left| h_{ij}^{(0)}(x) \right|^r f_i(x) m(\diff x)<\infty, 
\end{align*}
which is satisfied if $\varrho^{(j)} < \infty$ and the following holds.
\begin{condition} \label{condition_convergence_h_l_r_3} For given $i \in
  \M$, $j\in \M$, and $r \geq 1$, suppose that \eqref{cond_f_i_j_0} holds.
\end{condition}

On the other hand, by Lemma \ref{lemma_limit_l}, $\epsilon_n(i,j)/n
  \rightarrow 0$ as $n \uparrow \infty$ in $L^r(\P_i)$ under Condition \ref{condition_convergence_l_r_epsilon_2} below for $j \in \M \setminus \{i\}$, and, for $j = 0$, $\epsilon_n^{(0)}(i,j)/n
  \rightarrow 0$ as $n \uparrow \infty$ in $L^r(\P_i)$ under Condition \ref{condition_convergence_l_r_epsilon_3} below for $j \in \M$.  Notice as in Lemma \ref{lemma_limit_l} (vi) that in order for $L_n^{(i)}$ to converge in $L^r$ under $\P_i$ to zero, it is sufficient to have
\begin{align}
 \int_E \frac {f^{(0)}_i (x)} {f_i(x)} f^{(0)}_i(x) m (\diff x) < \infty, \label{condition_l_r_l_i_2}
\end{align}
because $\int_E \frac {f^{(0)}_i (x)} {f_i(x)} f_i(x) m (\diff x) = \int_E f^{(0)}_i (x)  m (\diff x)= 1 < \infty$.  

\begin{condition} \label{condition_convergence_l_r_epsilon_2} Given $i
  \in \M$, $j\in \M \setminus\{i\}$ and $r \geq 1$, we suppose that
  (\ref{condition_l_r_l_i_2}) holds,
\begin{enumerate}
\item if $j \in \Gamma_i$,
  (\ref{eq:sufficient-for-UI-of-L-average_2}) holds, and
\item if $j \in \M \setminus \Gamma_i$, (\ref{eq:sufficient-for-UI-of-K-average_2})   holds for the given $r$.
\end{enumerate}
\end{condition}

\begin{condition} \label{condition_convergence_l_r_epsilon_3} Given $i
  \in \M$, we suppose that (\ref{condition_l_r_l_i_2}) holds
 and $\max_{j \in \M}\varrho^{(j)} < \infty$ holds.
\end{condition}



In summary, we have the following $L^r$-convergence results.

\begin{proposition} \label{proposition_convergence_llr} 
\begin{enumerate}
\item 
For every $i
  \in \M$ and $j\in \M \setminus\{i\}$, we have $\Lambda_n(i,j)/n \rightarrow l(i,j)$ as $n \uparrow \infty$ in $L^r(\P_i)$ for some $r \geq 1$ if
Conditions \ref{condition_convergence_h_l_r_2} and
\ref{condition_convergence_l_r_epsilon_2} hold for the given $r$, 
\item For every $i
  \in \M$,
we have $\Lambda_n(i,0)/n \rightarrow l(i,0)$ as $n \uparrow \infty$ in $L^r(\P_i)$ for some $r \geq 1$ if Condition \ref{condition_convergence_h_l_r_3} holds for every $j \in \M$ and Condition \ref{condition_convergence_l_r_epsilon_3} holds.
\end{enumerate}


\end{proposition}

\subsection{Higher-Order approximations} \label{section_higher_approximation}

For Examples 1 and 2 described in Sections  \ref{subsection_example1} and \ref{subsection_example_2}, respectively, higher-order asymptotic approximations for
the minimum Bayes risk in Problem \ref{problem_bayes_risk} can be obtained by choosing appropriately
the values of $\sigma$ in \eqref{sigma_expression}.  Proposition
\ref{corollary_bound_kappa_i_and_r} (i) gives an upper bound on
$(R_i^{(a)} (\cdot,\cdot))_{i \in \M}$, and here we investigate if
there exists some $\sigma$ such that (\ref{convergence_sigma}) holds. 
 This can be obtained by a direct application of the theorems in \cite{Dayanik_2012}.

\begin{assumption} 
We assume that $j(i)$ is unique, for each $i \in \M$.
\end{assumption}
By Remarks \ref{remark_regarding_l} and \ref{remark_example_2},  $j(i) \in \{0\} \cup \Gamma_i $ in both Examples 1 and 2.   We shall first consider the case $j(i) \in \Gamma_i$.

\subsubsection{For the case $j(i) \in \Gamma_i$.} Because we are assuming $\Y_j := \{j\}$ for every $j \in \M$, we can set
\begin{align*}
\underline{a}_{ji} := \left\{ \begin{array}{ll} \min_{y \in \Y_0}a_{yi}, & j =0 \\ a_{ji}, & j \in \M \end{array} \right\} \quad \textrm{and} \quad \overline{a}_{ji} := \left\{ \begin{array}{ll} \max_{y \in \Y_0}a_{yi}, & j =0 \\ a_{ji}, & j \in \M \end{array} \right\},
\end{align*}
and for every $n \geq 1$
\begin{align*}
    \overline{G}_i^{(a)} (n) &:= \sum_{j \in \M_0 \setminus\{i\}} \overline{a}_{ji}
    e^{-\Lambda_n(i,j)} \quad \textrm{and} \quad \underline{H}_i^{(a)}(A_i) := - \log \overline{G}_i^{(a)}(\tau^{(i)}_A)+\log
        A_i - \log 1_{\left\{d_A=i, \; \theta \leq \tau_A <
          \infty\right\}}, \\
 \underline{G}_i^{(a)} (n) &:= \sum_{j \in \M_0 \setminus\{i\}} \underline{a}_{ji}
    e^{-\Lambda_n(i,j)} \quad \textrm{and} \quad  \overline{H}_i^{(a)}(A_i) := - \log \underline{G}_i^{(a)}(\tau^{(i)}_A)+\log
        A_i - \log 1_{\left\{d_A=i, \; \theta \leq \tau_A <
          \infty\right\}},
    \end{align*} 
    where it can be shown that $\overline{H}_i^{(a)}(A_i)$ and
    $\underline{H}_i^{(a)}(A_i)$ are bounded from below as in the
    proof of Lemma 5.1 of \cite{Dayanik_2012}.

Fix $i \in \M$. By Lemma \ref{lemma_changing_measure}  and because $\tau_A =
\tau^{(i)}_A$ on $\{d_A = i, \theta \le \tau_A < \infty\}$, we have
\begin{align*}
  {R_i^{(a)} (\tau_A,d_A)} / {A_i} &\leq \E_i
  \big[ 1_{\{d_A=i, \; \theta \leq \tau_A < \infty \}}
    \overline{G}_i^{(a)}(\tau^{(i)}_A) / A_i\big]
  = \E_i \big[ \exp \big\{- \underline{H}_i^{(a)}(A_i) \big\}\big], \\
  {R_i^{(a)} (\tau_A,d_A)} / {A_i} &\geq \E_i
  \big[ 1_{\{d_A=i, \; \theta \leq \tau_A < \infty \}}
    \underline{G}_i^{(a)}(\tau^{(i)}_A) / A_i\big]
  = \E_i \big[ \exp \big\{- \overline{H}_i^{(a)}(A_i) \big\}\big].
\end{align*}
Suppose  the overshoot  \begin{align}
W_i(A_i) := \Phi_{\tau^{(i)}_A}^{(i)} - (-\log A_i) = \Phi_{\tau^{(i)}_A}^{(i)} + \log A_i \geq 0, \label{def_overshoot}
\end{align}
converges in distribution under $\P_i$ to some random variable, say, $W_i$.  Then, as in Lemma 5.1 of \cite{Dayanik_2012}, $\underline{H}_i^{(a)} (A_i)$ and $\overline{H}_i^{(a)} (A_i)$ converge in distribution as $A_i \downarrow 0$ under $\P_i$ to
$W_i - \log a_{j(i)i}$ (note $a_{j(i)i} = \overline{a}_{j(i)i} = \underline{a}_{j(i)i}$ by the assumption that $j(i) \in \Gamma_i$).
Now because $x \mapsto e^{-x}$ is continuous and bounded on $x \in [b,\infty]$ for any $b \in \R$, we have
${R_i^{(a)} (\tau_A,d_A)} / {A_i} \xrightarrow{A_i \downarrow 0} \E_i [ \exp \{- W_i + \log a_{j(i)i} \}] = a_{j(i)i} \E_i [ \exp \{- W_i  \}]$,
and therefore (\ref{convergence_sigma}) holds with $\sigma_i = a_{j(i)i} \E_i [ \exp \{- W_i  \}]$. 

\begin{lemma} \label{lemma_sigma} Fix $i \in \M$. If $j(i) \in \Gamma_i$ is unique
  and the overshoot $W_i(A_i)$ in (\ref{def_overshoot}) converges in
  distribution as $A_i \downarrow 0$ to some random variable $W_i$
  under $\P_i$, then (\ref{convergence_sigma}) holds with
  $\sigma_i := a_{j(i)i} \E_i [ \exp \{- W_i \}]$.
\end{lemma}

Now we obtain the limiting distribution of \eqref{def_overshoot}.
Similarly to \cite{Dayanik_2012}, we have a decomposition $\Phi_n^{(i)} =
\sum_{l=\theta \vee 1}^n h_{i j(i)} (X_l) + \xi_n(i,j(i))$, where, for $n \geq 1$,
\begin{align*}
  \xi_n(i,j(i)) &:= \sum_{l=1}^{n \wedge (\theta -1)} h_{ij(i)}(X_l) +
  \epsilon_n (i,j(i)) - \log \Big( 1 + \sum_{j \in \mathcal{M}_0 \setminus
    \{i,j(i) \}} \exp (\Lambda_n(i,j(i)) - \Lambda_n(i,j)) \Big).
\end{align*}
By Lemmas  \ref{lemma_limit_l} and \ref{lemma_limit_2} and because the last term of the right-hand side converges to zero $\P_i-a.s.$, the remaining term $\xi_n(i,j(i))$ converges to a finite random variable, and hence is \emph{slowly changing} (cf. Definitions 5.2 and 5.3 of \cite{Dayanik_2012}).  This allows us to apply nonlinear renewal theory.


Define a stopping time,  $T_i := \inf \big\{ n \geq 1: \sum_{l=1}^{n} h_{ij(i)}(X_l)> 0\big\}$,
and random variable $W_i$ whose distribution is given by
\begin{align}
  \P_i \{ W_i\leq w \} = \frac {\int_1^w \P_i^{(0)} \big\{
      \sum_{l=1}^{T_i} h_{ij(i)}(X_l) >
      s\big\} \diff s}{\E_i^{(0)} \big[ \sum_{l=1}^{T_i}h_{ij(i)}(X_l)\big]},\qquad 0 \leq w <
  \infty. \label{dist_overshoots}
\end{align}

\begin{proposition} \label{proposition_asymptotic_kappa} Fix $i \in
  \M$ and suppose $j(i) \in \Gamma_i$ is unique. Then ${R_i^{(a)} (\tau_A,d_A)} /
  {A_i} \xrightarrow{A_i \downarrow 0} a_{j(i) i} \E_i [ e^{-W_i}]$. Therefore, a higher-order approximation for
  Problem \ref{problem_bayes_risk} can be achieved by setting in
  (\ref{g_minimizer}), $\sigma_i := a_{j(i) i} \E_i \left[ e^{- W_i}\right]$. 

\end{proposition}

\subsubsection{For the case $j(i) = 0$.}  Now suppose $j(i) = 0$ and is unique.   As in Remarks  \ref{remark_regarding_l} (4)  and \ref{remark_example_2} (4),  $\varrho^{(j)}$ and $q^{(0)}(i,j) + \varrho^{(j)}$ in Examples 1 and 2, respectively, are minimized when $j=i$ and is unique. Here we assume that
\begin{align}
\begin{split}
&a_{yi} = a_{zi} =: a^{(0)}_i, \quad y,z \in \Y_0, \quad \textrm{ in Example 1}\\
&a_{yi} = a_{zi} =: a^{(0)}_i, \quad y,z \in \Y_i^{(0)},  \quad \textrm{ in Example 2}.
\end{split} \label{assumption_additional}
\end{align}
Similarly to the above, we have a decomposition: for every $n
  \geq 1$, 
\begin{align*}
\Phi_n^{(i)} &= \left\{ \begin{array}{ll} \sum_{l=\theta \wedge 1}^n   h_{i0}(X_l) + \sum_{l=1}^{n\wedge (\theta-1)} h_{i0}(X_l) + \epsilon_n(i,0) + \eta_n^{(0)}(i), & \textrm{in Example 1}, \\
 \sum_{l=\theta \wedge 1}^n   h^{(0)}_{ii}(X_l) + \sum_{l=1}^{n\wedge (\theta-1)} h^{(0)}_{ii}(X_l) + \epsilon_n^{(0)}(i,i) + \eta_n^{(0)}(i), & \textrm{in Example 2},
\end{array} \right.
\end{align*}
where
\begin{align*}
  \eta_n^{(0)}(i) &:= \left\{ \begin{array}{ll} \log \Big( 1 + \sum_{j
        \in \mathcal{M} \setminus
        \{i\}} \exp (\Lambda_n(i,0) - \Lambda_n(i,j)) \Big), & \textrm{in Example 1},\\
      \log \Big( 1 + \sum_{j \in \mathcal{M} \setminus \{i \}} \exp
      (\Lambda_n^{(0)}(i,i) - \Lambda_n(i,j)) & \\
      \hspace{8em}+ \sum_{j \in \mathcal{M} \setminus \{i\}} \exp
      (\Lambda_n^{(0)}(i,i) - \Lambda_n^{(0)}(i,j)) \Big), &
      \textrm{in Example 2}.
    \end{array} \right. 
\end{align*}
Here, under $\P_i$, $\sum_{l=1}^{n\wedge (\theta-1)} h_{i0}(X_l)$ and $\sum_{l=1}^{n\wedge (\theta-1)} h^{(0)}_{ii}(X_l)$ are finite a.s., $\eta_n^{(0)}(i)$ converges to zero a.s. in view of the limits as in \eqref{definition_limit_llr_2}, \eqref{definition_limit_llr_3} and the fact that in Example 2 $\Lambda_n^{(0)}(i,j)/n \rightarrow q^{(0)}(i,j)+\varrho^{(j)}$ (which is uniquely minimized when $j=i$).  Hence these terms are slowly-changing.  

 It remains to show that  $\epsilon_n(i,0)$ (resp.\  $\epsilon_n^{(0)}(i,i)$ ) is slowly-changing for Example 1 (resp. Example 2).  In view of Assumption \ref{assumption_rho} and Lemmas \ref{lemma_limit_l} and  \ref{lemma_limit_2}, it holds on condition that the following holds. Notice in Example 1 that
 \begin{align*}
 \log \Big[ \sum_{j \in \M} \nu_j \big(1- \sum_{t=0}^n \rho_t^{(j)} \big)  \Big] = \log  \big(1- \sum_{t=0}^n \rho_t^{(i)} \big) + \log \nu_i + \log \Big[ 1 + \sum_{j \in \M \backslash \{i\}} \frac {\nu_j} {\nu_i} \frac {1- \sum_{t=0}^n \rho_t^{(j)} }   {1- \sum_{t=0}^n \rho_t^{(i)} } \Big],
\end{align*}
 where the last term converges to zero by Assumption \ref{assumption_rho} and is hence slowly-changing. 
\begin{assumption}   \label{assump_ucip}For both Examples 1 and 2, we assume $\zeta_n^{(i)} := -{\log \big( 1- \sum_{t=0}^n \rho_t^{(i)}  \big)} - n \varrho^{(i)}$ is uniformly continuous in probability, i.e., for any $\varepsilon > 0$, there exists $\delta > 0$ such that 
\begin{align*}
\max_{0 \leq k \leq n \delta} |\zeta_{n+k}^{(i)}-\zeta_n^{(i)}| < \varepsilon, \quad \textrm{ for all } n \geq 1.
\end{align*} 
\end{assumption}
 Let  $\widetilde{T}_i := \inf \big\{ n \geq 1: \sum_{l=1}^{n} h_{i0}(X_l)> 0\big\}$ (resp.\ $\widetilde{T}_i := \inf \big\{ n \geq 1: \sum_{l=1}^{n} h_{ii}^{(0)}(X_l)> 0\big\}$) for Example 1 (resp. Example 2),
and the distribution of random variable $\widetilde{W}_i$  is given by
\begin{align*}
  \P_i \{ \widetilde{W}_i\leq w \} = \left\{ \begin{array}{ll} \frac {\int_0^w \P_i^{(0)} \big\{
      \sum_{l=1}^{T_i} h_{i0}(X_l) >
      s\big\} \diff s}{\E_i^{(0)} \big[ \sum_{l=1}^{T_i}h_{i0}(X_l)\big]}, & \textrm{in Example 1}\\ \frac {\int_0^w \P_i^{(0)} \big\{
      \sum_{l=1}^{T_i} h_{ii}^{(0)}(X_l) >
      s\big\} \diff s}{\E_i^{(0)} \big[ \sum_{l=1}^{T_i}h_{ii}^{(0)}(X_l)\big]}, &  \textrm{in Example 2} \end{array} \right\}, \qquad 0 \leq w <
  \infty. 
\end{align*}
Following the same arguments as in the case $j(i) \in \Gamma_i$, we have the following.
\begin{proposition} \label{proposition_asymptotic_kappa} Fix $i \in
  \M$ and suppose $j(i)=0$ is unique. Moreover, suppose \eqref{assumption_additional} and Assumption \ref{assump_ucip} hold. Then ${R_i^{(a)} (\tau_A,d_A)} /
  {A_i} \xrightarrow{A_i \downarrow 0} a_{i}^{(0)} \E_i [ e^{-\widetilde{W}_i}]$. Therefore, a higher-order approximation for
  Problem \ref{problem_bayes_risk} can be achieved by setting in
  (\ref{g_minimizer}), $\sigma_i := a_{i}^{(0)} \E_i \left[ e^{- \widetilde{W}_i}\right]$. 

\end{proposition}

\section{Numerical Examples} \label{section_numerics}
In this section, we verify the effectiveness of the asymptotically optimal strategies through a series of numerical experiments.   Because the optimality results are fundamentally relying on the existence of the limits $l(i,j)$ as in Assumption \ref{condition_a_s_convergence}, we first verify their existence numerically and show that they can be obtained efficiently via simulation.  We then evaluate the performance of the asymptotically optimal strategies  and also the rate of convergence.

%
%
%

\subsection{Verification of Assumption \ref{condition_a_s_convergence}} \label{numerics_convergence}
We consider both the case $X$ is i.i.d.\ in each of the closed sets as studied in Section \ref{section_convergence_results} and also the non-i.i.d.\ case where each closed set may contain multiple states.

In order to verify the convergence results in Section \ref{section_convergence_results}, we consider Example 2 of Subsection \ref{subsection_example_2} with $M=2$ and the hidden Markov chain  $\Y_1 = \{1 \}$, $\Y_0^{(1)} = \{ (1,1), (1,2)\}$, $\Y_2 = \{2 \}$, and $\Y_0^{(2)} = \{ (2,1), (2,2)\}$ with 
\begin{align*}
P = \begin{array}{c} (1,1) \\  (1,2) \\ 1 \\  (2,1) \\  (2,2) \\ 2 \end{array} \left[ \begin{array}{ll|l|ll|l} 
 .85 &  .15 &  0 &         0     &    0    &     0 \\
         0  &  .9  &  .1     &    0     &    0    &     0\\
\hline
         0    &     0  &  1   &      0      &   0     &    0 \\
\hline
         0   &      0     &    0 &   .8     &    0  &  .2 \\
         0     &    0     &    0     &    0  &  .95  & .05 \\
\hline
         0     &    0     &    0     &    0   &      0  &  1 \end{array}\right] \quad \textrm{and} \quad \eta = \left[ \begin{array}{c} .25 \\ .25 \\ 0\\ .25 \\ .25 \\ 0\end{array}\right].
\end{align*}
Under $\P_1$, $Y$ starts at either $(1,1)$ or $(1,2)$ and gets absorbed by $1$, while under $\P_2$ it starts at either $(2,1)$ or $(2,2)$ and gets absorbed by $2$.
Conditionally given $Y_0 = (1,1)$, the absorption time $\theta$  is a sum of two independent geometric random variables with parameters $0.15$ and $0.1$; conditionally on $Y_0 = (1,2)$,  it is  geometric with parameters $0.1$.  It is easy to show that the exponential tail \eqref{exponential_tail} under $\P_1$ is $\varrho^{(1)} = |\log (1-\min(0.1, 0.15))|$.  On the other hand, regarding $\Y_2 \cup \Y_0^{(2)}$, the absorption time $\theta$ is a mixture of two geometric random variables $0.2$ and $0.05$.  Its exponential tail is  $\varrho^{(2)} = |\log (1-\min (0.2, 0.05))|$.

For the observation process $X$, we assume that it is normally
distributed with a common variance $1$ and its conditional mean given
$Y$ is $\{ \lambda(y); y \in \Y \}$. As is assumed in Example 2, we let
$\lambda^{(0)}_1 := \lambda((1,1)) = \lambda((1,2))$ and
$\lambda_2^{(0)} := \lambda((2,1)) = \lambda((2,2))$.  We also let
$\lambda_k := \lambda(k)$ for $k=1,2$. The Kullback-Leibler divergence
is $q(i,j) = \big( \lambda_i - \lambda_j \big)^2/2$ for every
$i\in\M$, $j\in\M \setminus \{i\}$ and $q^{(0)}(i,j) = \big( \lambda_i
- \lambda^{(0)}_j\big)^2/2$ for every $i,j\in\M$.  Here we assume that
$\lambda^{(0)}_1 = 0.1$, $\lambda_1 = 0.7$, $\lambda^{(0)}_2 = 0$ and
$\lambda_2 = 0.2$.  Using Proposition \ref{prop_convergence_example_2},
the analytical limit values $l(i,j)$ are obtained and are listed in
the last column of Table \ref{tab:limits-for-numerical-example_2}.

In Figure \ref{figure_llr_2}, we plot  sample paths of $\Lambda_n(1,\cdot)/n$ under $\P_1$ and $\Lambda_n(2,\cdot)/n$ under $\P_2$ along with the theoretical limit $l(i,j)$.  In order to verify their almost sure convergence, we show in Table \ref{tab:limits-for-numerical-example_2} the statistics on the position at time $n=500, 1000, 1500$ based on $1000$ samples for each.   We indeed see that the mean value approaches the theoretical limit and the standard deviation diminishes as $n$ increases, verifying the almost sure limit of the LLR processes.

\begin{figure}[htp]
\begin{center}
\begin{tabular}{cc}
  \ifpdf
  \includegraphics[width=0.5\textwidth] {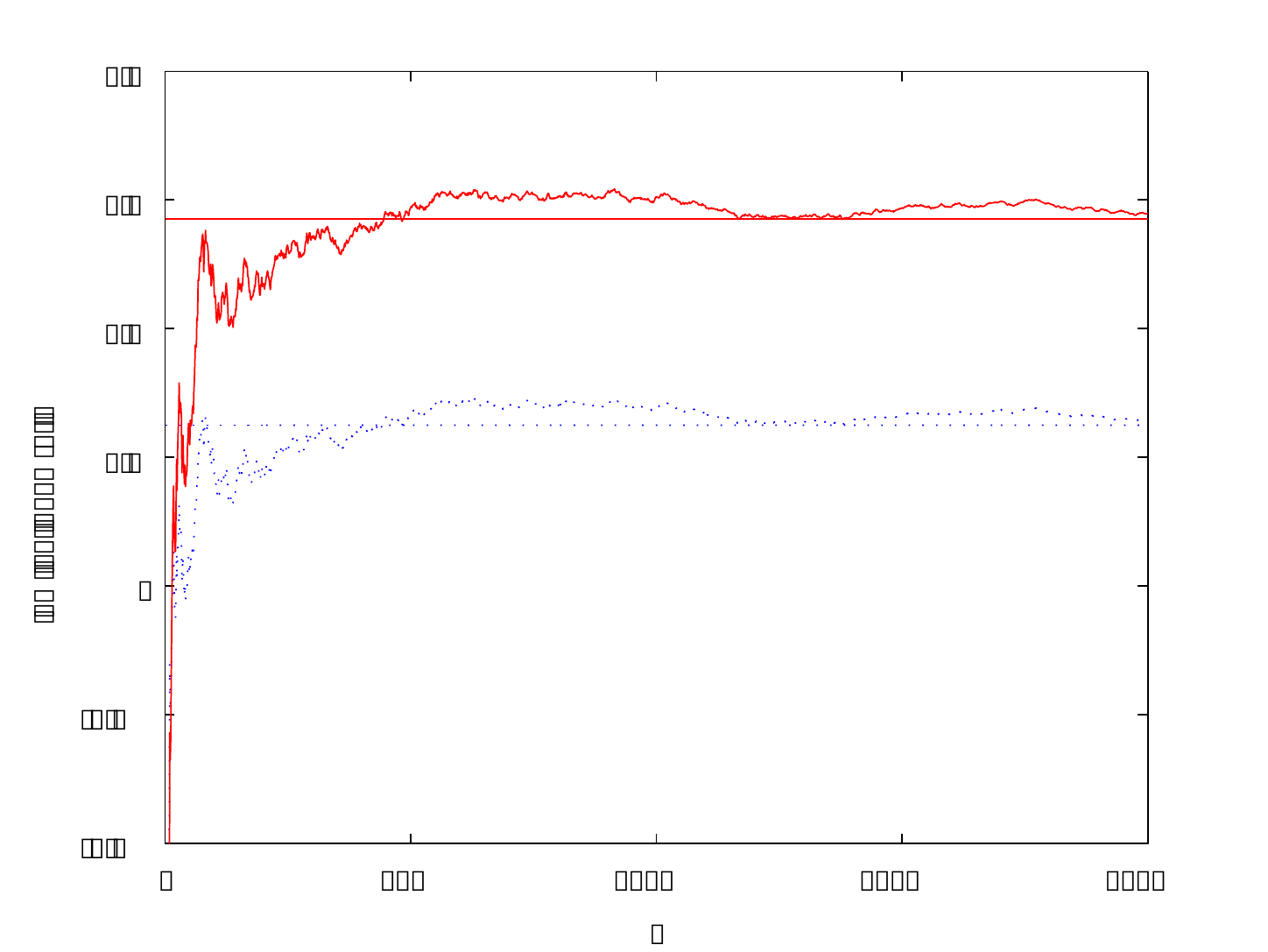} 
  &
  \includegraphics[width=0.5\textwidth] {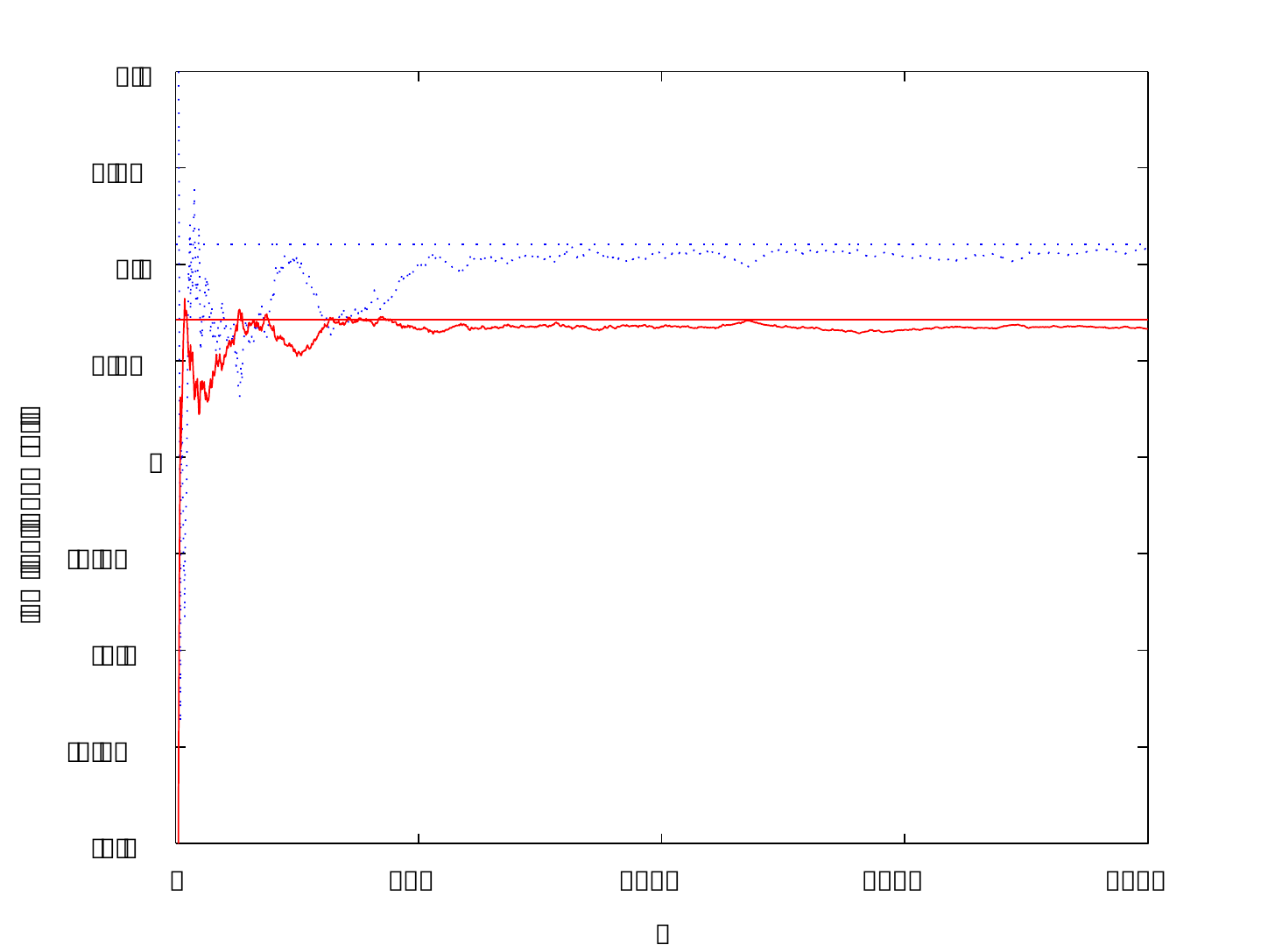} 
\\ 
(a) $\Lambda_n(1,\cdot)/n$ under $\P_1$ & (b) $\Lambda_n(2,\cdot)/n$ under $\P_2$  \\
\fi
\end{tabular}
\caption[Realizations of the LLR processes.]{Sample realizations
  of LLR Processes: (a) $\Lambda_n(1,0)/n$ (solid)  and $\Lambda_n(1,2)/n$ (dotted) under $\P_1$ and  (b) $\Lambda_n(2,0)/n$ (solid) and $\Lambda_n(2,1)/n$ (dotted)  under $\P_2$.  The theoretical limit values $l(\cdot, \cdot)$ are also given.
}
\label{figure_llr_2}
  \end{center}
\end{figure}

\begin{table}[!t]
  \centering
  \begin{tabular}[t]{c|ccc|c}
     n   & $500$ & $1000$ & $1500$ & theoretical values\\ \hline
    $\Lambda_n(1,0)$ under $\P_1$ &.2790 (.0272) & .2818 (.0193)   &  .2830 (.0154) & .2854\\
    $\Lambda_n(1,2)$ under $\P_1$  & .1218 (.0225) & .1231 (.0161) & .1238 (.0128) &  .1250\\
     $\Lambda_n(2,0)$ under $\P_2$ & .0721 (.0084) &    .0715 (.0062)   & .0714 (.0051)  & .0713\\
    $\Lambda_n(2,1)$ under $\P_2$  & .0948 (.0096)& .1006 (.0063) & .1032 (.0048)& .1104\\
  \end{tabular}\vspace*{0.5em}
  \caption{The LLR process at time $n = 500, 1000,1500$: mean and standard deviation along with theoretical values.}
  \label{tab:limits-for-numerical-example_2}
\end{table}


We now consider the non-i.i.d.\  case where each closed set consists of multiple states.  Because this case has not been covered in Section \ref{section_convergence_results} and the limit $l(i,j)$ has not been derived, we shall confirm this numerically via simulation. We consider a Markov chain with $M = 2$, $\Y_0 = \{0\}$,  $\Y_1 = \{(1,1),(1,2),(1,3)\}$ and  $\Y_2 = \{(2,1),(2,2)\}$.  We consider two cases with transitional matrices:
\begin{align*}
P_1 := \begin{array}{c} 0\\  (1,1) \\ (1,2) \\  (1,3) \\  (2,1) \\ (2,2) \end{array}  \left[ \begin{array}{l|lll|ll} 
.75 & .05 & .05 & .05 & .05 & .05 \\
\hline
    0 & .5 & .2 & .3 & 0&  0 \\
    0& .3 &.5 & .2 & 0 & 0 \\
    0 & .3 & .2 & .5 & 0 & 0 \\
\hline
    0 & 0 & 0 & 0 & .7 & .3 \\
    0 & 0 & 0 & 0 & .2 & .8 \end{array}\right], \quad  P_2 := \left[ \begin{array}{l|lll|ll} 
.75 & .05 & .05 & .05 & .05 & .05 \\
\hline
    0 & 0 & 1 & 0 & 0&  0 \\
    0& 0 &0 & 1 & 0 & 0 \\
    0 & 1 & 0 & 0 & 0 & 0 \\
\hline
    0 & 0 & 0 & 0 & 0 & 1 \\
    0 & 0 & 0 & 0 & 1 & 0 \end{array}\right].
\end{align*}
Here we model the acyclic case for the former and cyclic case for the latter.  For both cases, we assume the initial distribution $\eta = \left[ 1, 0, 0, 0, 0,0 \right]$ and $X$ is again normally distributed with variance $1$ and mean function $\lambda = \left[  0, 0.2, 0.4, 0.6, -0.2, -0.4  \right]$.


We plot in Figure \ref{figure_llr} sample paths of the LLR processes $\Lambda_n(1,\cdot)/n$ under $\P_1$ and $\Lambda_n(2,\cdot)/n$ under $\P_2$ and also show in Table \ref{tab:limits-for-numerical-example} the statistics on their positions at $n = 500,1000,1500$ based on $1000$ sample paths.  We observe that these processes indeed converge to deterministic limits almost surely.  In fact due to the simple structure of the transient set $\Y_0$, the convergence seems to be faster than what are observed in Figure \ref{figure_llr_2} and Table \ref{tab:limits-for-numerical-example_2}.  It is also noted that the convergence holds regardless of the cyclic/acyclic structure of the closed sets.


\begin{figure}[htp]
\begin{center}
\begin{tabular}{cc}
  \ifpdf
  \includegraphics[width=0.5\textwidth] {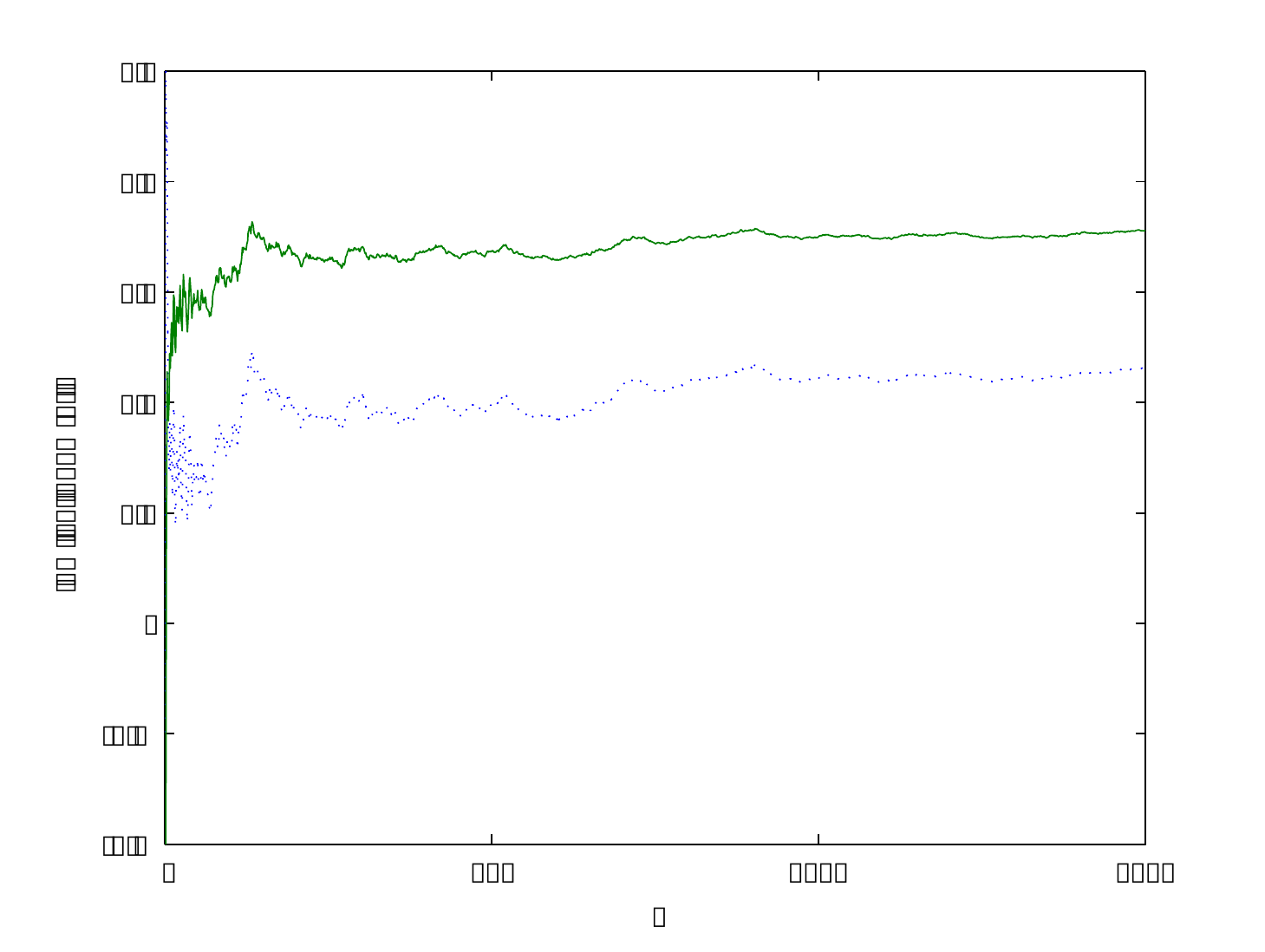} 
  &
  \includegraphics[width=0.5\textwidth] {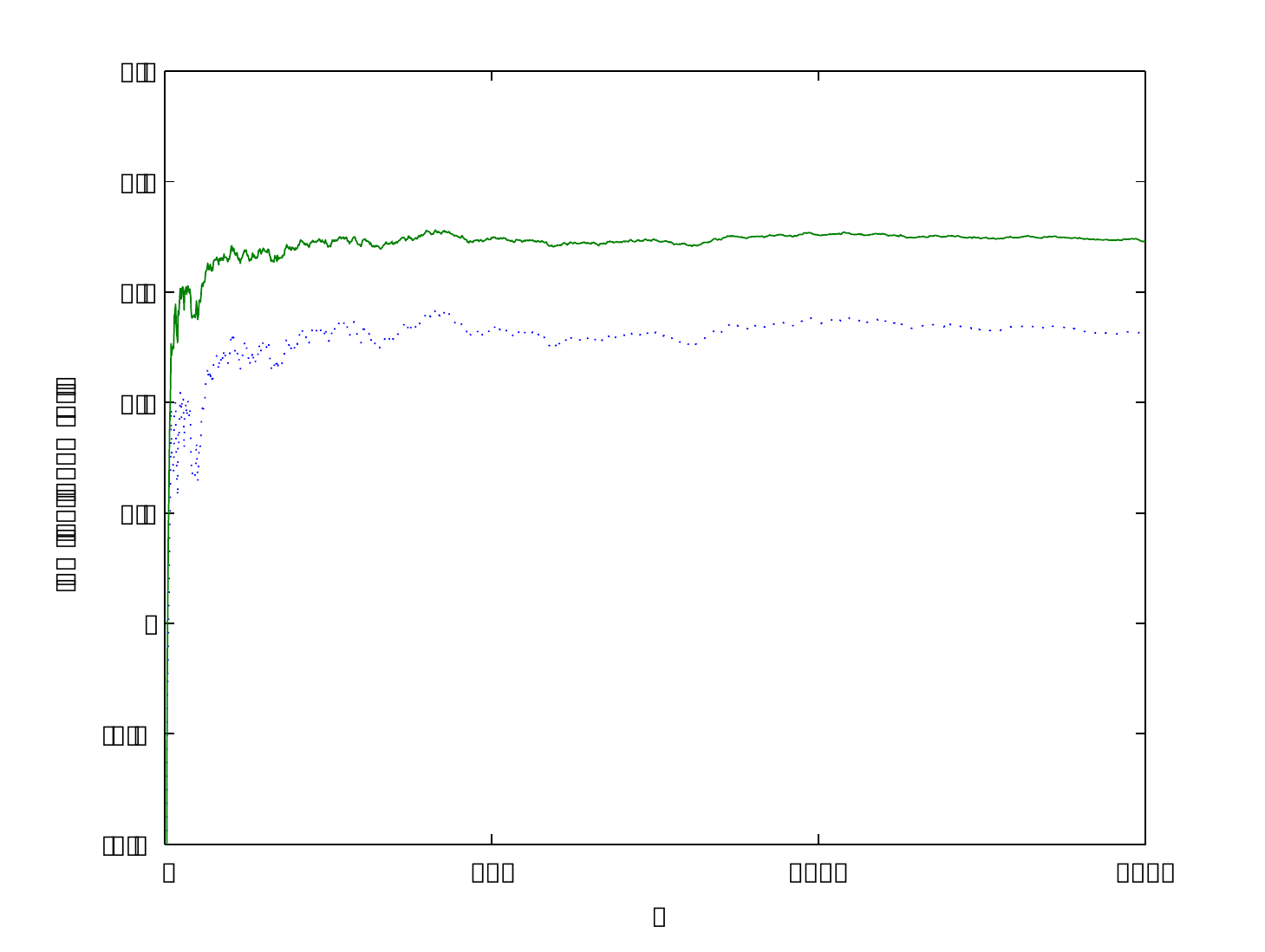} 
\\ 
(a-1) case 1 under $\P_1$ & (b-1) case 1 under $\P_2$ \\
  \includegraphics[width=0.5\textwidth] {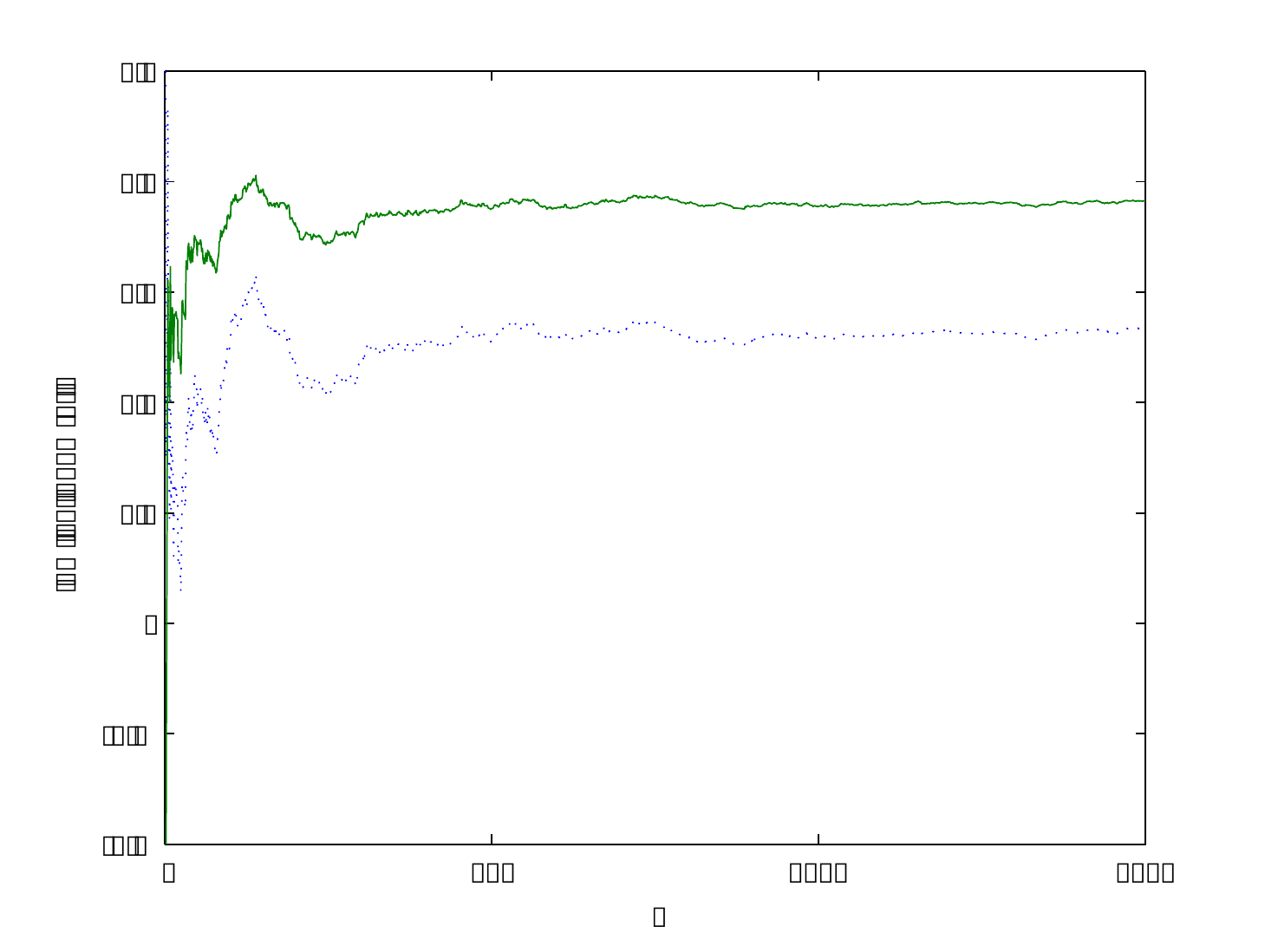} 
  &
  \includegraphics[width=0.5\textwidth] {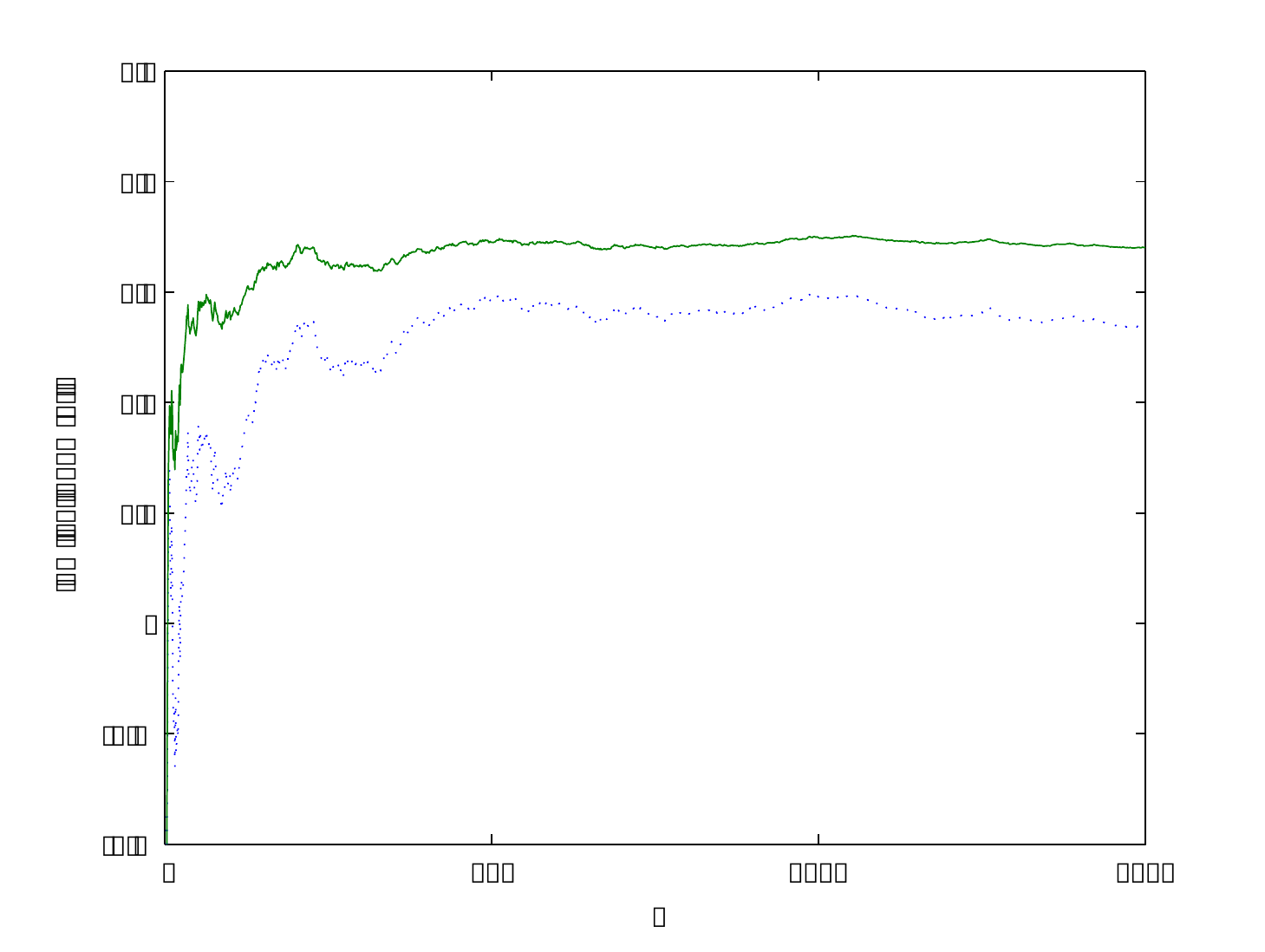} 
\\ 
(a-2) case 2 under $\P_1$ & (b-2) case 2 under $\P_2$ \\
\fi
\end{tabular}
\caption[Realizations of the LLR processes.]{Sample realizations
  of LLR Processes: (a) $\Lambda_n(1,0)/n$ (solid)  and $\Lambda_n(1,2)/n$ (dotted) under $\P_1$ and  (b) $\Lambda_n(2,0)/n$ (solid) and $\Lambda_n(2,1)/n$ (dotted)  under $\P_2$.
}
\label{figure_llr}
  \end{center}
\end{figure}

\begin{table}[!t]
  \centering
  \begin{tabular}[t]{c|ccc}
        & $500$ & $1000$  & $1500$\\ \hline
    $\Lambda(1,0)$ under $\P_1$ & .3636 (.0171)& .3641 (.0130)& .3639 (.0104)    \\
    $\Lambda(1,2)$ under $\P_1$  & .2451 (.0299) & .2456 (.0228) & .2450 (.0182)  \\
     $\Lambda(2,0)$ under $\P_2$ & .3364 (.0139) & .3376 (.0104) & .3375 (.0085)         \\
    $\Lambda(2,1)$ under $\P_2$  & .2393 (.0293)& .2415 (.0221) & .2412 (.0181)\\
  \end{tabular} \\
case 1 \\
  \begin{tabular}[t]{c|ccc}
        & $500$ & $1000$  & $1500$\\ \hline
    $\Lambda(1,0)$ under $\P_1$ & .3775 (.0188)& .3801 (.0133)& .3804 (.0112)     \\
    $\Lambda(1,2)$ under $\P_1$  & .2575 (.0313) & .2610 (.0220) &  .2614 (.0186)  \\
     $\Lambda(2,0)$ under $\P_2$ & .3340 (.0146) & .3362 (.0101) &.3361 (.0081)        \\
    $\Lambda(2,1)$ under $\P_2$  & .2508 (.0324) & .2564 (.0224) & .2567 (.0182) \\
  \end{tabular} \\ case 2 \vspace*{0.5em}
  \caption{The LLR process at time $n = 500, 1000,1500$: mean and standard deviation.}
  \label{tab:limits-for-numerical-example}
\end{table}

%


\subsection{Numerical results on asymptotic optimality}  We now evaluate the asymptotically optimal strategy in comparison with the optimal Bayes risk focusing on Problem \ref{problem_bayes_risk} with $m=1$.  \cite{Dayanik2009}  showed that the problem can be reduced to an optimal stopping problem of the posterior probability process $\Pi$, and in theory the value function can be approximated via value iteration in combination with discretization.  In practice, however, the state space increases exponentially in the number of states $|\Y|$, and 
it is computationally feasible only when $|\Y|$ is small (typically at most three or four).  Moreover, we need to deal with small detection delay costs $c$ and hence the resulting stopping regions tend to be very small in practical applications.  For this reason, the approximation is affected severely by discretization errors as well.  Here in order to provide reliable approximation to the optimal Bayes risk, we consider the following simple examples.

We suppose $M = 2$, $\Y_0 = \{ (0,1), (0,2)\}$,  $\Y_1 = \{1\}$ and $\Y_2 = \{2\}$ and consider case 1 with
\begin{align*}
P_1 := \begin{array}{c} (0,1)\\  (0,2) \\ 1 \\  2 \end{array} \left[ \begin{array}{llll} 
.95 & 0 & .05 & 0 \\
0 & .85 & 0  & .15 \\
0 & 0 & 1 & 0 \\
0 & 0 & 0 & 1
 \end{array}\right] \quad \textrm{and} \quad \eta_1 := \left[ \begin{array}{c} .5 \\ .5 \\ 0 \\ 0 \end{array} \right]
\end{align*}
and case 2 with
\begin{align*}
P_2 := \begin{array}{c} (0,1)\\  (0,2) \\ 1 \\  2 \end{array} \left[ \begin{array}{llll} 
.95 & .05 & 0 & 0 \\
0 & .85 & .05  & .1 \\
0 & 0 & 1 & 0 \\
0 & 0 & 0 & 1
 \end{array}\right] \quad \textrm{and} \quad \eta_2 := \left[ \begin{array}{c} 1 \\ 0 \\ 0 \\ 0 \end{array} \right].
\end{align*}
Case 1 has been considered in \cite{Dayanik2009} where $\theta$ is geometric with parameter $.05$ under $\P_1$ and $.15$ under $\P_2$. In Case 2, it is  a sum of two geometric random variables under $\P$.
For $X$, we assume for both cases that it takes values in $E  = \{1,2,3,4\}$ with probabilities $\P\{ X_1 = k | Y_1=y\}=f(y,k)$ given by
\begin{align*}
f = \left[ \begin{array}{llll} 
.25 & .25 & .25 & .25 \\
.25 & .25 & .25  & .25 \\
.4 & .3 & .2 & .1 \\
.1 & .2 & .3 & .4
 \end{array}\right].
\end{align*}


\begin{table}[!t]
  \centering
  \begin{tabular}[t]{c|c|c|c}
      c  & asymptotic & optimal  & ratio\\ \hline
    $.5$  & 1.45357 (1.44847,1.45867)&1.02350 (1.01881,1.02819)&  1.42020\\
    $.1$  &1.01413 (1.01106,1.01719)& 0.80195 (0.79510,0.80880)  &  1.26458 \\
    $.05$  & 0.72380 (0.72149,0.72611)& 0.62557 (0.61869,0.63245)& 1.15702 \\
    $.01$ & 0.25023 (0.24907,0.25139)& 0.24226 (0.23763,0.24690) &  1.03288 \\
    $.005$ & 0.14843 (0.14756,0.14929) & 0.14440 (0.14106,0.14775) &  1.02787 \\
  \end{tabular} \\ case 1 \\ 
  \begin{tabular}[t]{c|c|c|c}
      c  & asymptotic & optimal  & ratio\\ \hline
    $.5$  & 1.61202 (1.60403,1.62000) & 1.01375 (1.01099,1.01651)&  1.59015 \\
    $.1$  & 1.12962 (1.12617,1.13307)&  0.91009 (0.90408,0.91610) & 1.24122\\
    $.05$  & 0.81023 (0.80785,0.81261) & 0.73136 (0.72473,0.73800) & 1.10783  \\
    $.01$ & 0.27809 (0.27684,0.27933) & 0.27454 (0.27011,0.27896) & 1.01293 \\
    $.005$ & 0.16287 (0.16194,0.16380) & 0.16269 (0.15893,0.16644) & 1.00115 \\
  \end{tabular} \\ case 2
\vspace*{0.5em}
  \caption{Comparison with the optimal value function.}
  \label{tab:optimal}
\end{table}

We set the detection delay function $c = [0, 0, \bar{c}, \bar{c}]$ and the terminal decision loss function $a_{yi} = 1$ for $y \notin \Y_i$ and it is zero otherwise.  The limits $l(i,j)$ can be analytically computed by Propositions \ref{prop_convergence_example_1} and the asymptotically optimal strategy can be constructed analytically.  Here we set the value $\sigma_i = a_{j(i)i} = 1$ and hence $A_i(c) = c_i /l(i)$, for every $i \in \M$.
 In order to compute the optimal Bayes risk, we first discretize the state space of $\Pi$ ($|\Y|-1$-simplex) by $70^{|\Y|-1}$ mesh and then obtain the stopping regions by solving the optimality equation provided in \cite{Dayanik2009}  via value iteration.  The optimal Bayes risk is then approximated via simulation based on $10,000$ paths.   The risk under the asymptotically optimal strategy is approximated based on $100,000$ paths. 

Table \ref{tab:optimal} shows the results.  It shows the approximated Bayes risk (with 95\% confidence interval) for both strategies and also the ratio between the two.    It can be seen that the ratio indeed converges to $1$.  In fact, the results show that the convergence is fast and it approximates the optimal Bayes risk precisely even for a moderate value of $\bar{c}$.  The proposed strategy can be derived analytically and its corresponding Bayes risk can be computed instantaneously via simulation.


\section*{Acknowledgments} Savas Dayanik was supported by the
T{\"U}B{\.I}TAK Research Grant 109M714. Kazutoshi Yamazaki was partially supported by Grant-in-Aid for Young Scientists (B) No.\ 22710143, the Ministry of Education, Culture, Sports, Science and Technology, and  by Grant-in-Aid for Scientific Research (B) No.\  2271014, Japan Society for the Promotion of Science.

\appendix
\section{Proofs}

\subsection{Proof of Lemma \ref{lemma_lower_bound}}
The proof of Lemma \ref{lemma_lower_bound} requires the following lemmas. 
\begin{lemma} \label{lemma_lower_bound_2}
For every $i \in \M$ and $j \in \M_0 \setminus \{i\}$, $L > 0$, $\gamma > 0$ and $k > 1$, we have
\begin{multline*}
\inf_{(\tau,d) \in \overline{\Delta}(\overline{R})}\P_i \Big\{ \sum_{m=1}^{\tau-1} c(Y_m) >  \gamma L \Big\} \\ \geq 1 - \frac {\sum_{y \in \Y \setminus \Y_j} \overline{R}_{yi}} {\nu_i}  - \frac {e^{k L l(i,j)}} {\nu_i} \sum_{y \in \Y_j}\overline{R}_{yi} - \P_i \Big\{ \sup_{n \leq \theta + L} \Lambda_n (i,j) > k L l(i,j) \Big\} - \P_i \Big\{ \min_{n \geq L}\frac {\sum_{m=1}^{n-1} c(Y_m)} {n} < \gamma \Big\}.
\end{multline*}
\end{lemma}
\begin{proof}
We have
\begin{align*}
\P_i \Big\{ {\sum_{m=1}^{\tau-1} c(Y_m)}  > \gamma L \Big\} \geq \P_i \Big\{ \frac {\sum_{m=1}^{\tau-1} c(Y_m)} {\tau} \geq \gamma, \tau > L \Big\} = \P_i \left\{ \tau > L \right\} - \P_i \Big\{ \frac {\sum_{m=1}^{\tau-1} c(Y_m)} {\tau} < \gamma, \tau > L \Big\}.
\end{align*}
Moreover, we have
\begin{multline*}
\P_i \Big\{ \frac {\sum_{m=1}^{\tau-1} c(Y_m)} {\tau} < \gamma, \tau > L \Big\} \leq \P_i \Big\{ \inf_{n \geq \tau}\frac {\sum_{m=1}^{n-1} c(Y_m)} {n} < \gamma, \tau > L \Big\} \\ \leq \P_i \Big\{ \inf_{n \geq L}\frac {\sum_{m=1}^{n-1} c(Y_m)} {n} < \gamma, \tau > L \Big\} \leq \P_i \Big\{ \inf_{n \geq L}\frac {\sum_{m=1}^{n-1} c(Y_m)} {n} < \gamma \Big\}.
\end{multline*}
As in the proof of Lemma A.1 of \cite{Dayanik_2012}, 
\begin{align*}
\P_i \left\{ \tau  > L \right\} \geq \P_i \left\{ \tau-\theta  > L \right\} \geq 1 - R_i^{(1)} (\tau,d) - \frac {e^{k L l(i,j)}} {\nu_i} \widetilde{R}_{ji} (\tau,d) - \P_i \big\{ \sup_{n \leq \theta + L} \Lambda_n (i,j) > k L l(i,j) \big\}.
\end{align*}
Combining the above and take infimum over $\overline{\Delta} (\overline{R})$,
\begin{multline*}
\inf_{(\tau,d) \in \overline{\Delta}(\overline{R})} \P_i \Big\{ \sum_{m=1}^{\tau-1} c(Y_m) > \gamma L  \Big\}
\geq 1 - \sup_{(\tau,d) \in \overline{\Delta}(\overline{R})}R_i^{(1)} (\tau,d) - \frac {e^{k L l(i,j)}} {\nu_i} \sup_{(\tau,d) \in \overline{\Delta}(\overline{R})} \widetilde{R}_{ji} (\tau,d) \\ - \P_i \Big\{ \sup_{n \leq \theta + L} \Lambda_n (i,j) > k L l(i,j) \Big\} - \P_i \Big\{ \min_{n \geq L} \frac {\sum_{m=1}^{n-1} c(Y_m)} n < \gamma \Big\}.
\end{multline*}
Therefore the lemma holds because $(\tau,d) \in \overline{\Delta}(\overline{R})$ implies that $R_i^{(1)}(\tau,d) \leq {\sum_{y \in \Y \setminus \Y_j} \overline{R}_{yi}/\nu_i}$ and $\widetilde{R}_{ji}(\tau,d) \leq \sum_{y \in \Y_j}\overline{R}_{yi}$.
\end{proof}

\begin{lemma} \label{lemma_for_lower_bound}
Fix $0 < \delta < 1$, $i \in \M$ and $j(i)$. We have
\begin{align*}
&\liminf_{ \overline{R}_i \downarrow 0} \inf_{(\tau,d) \in \overline{\Delta}
(\overline{R})} \P_i \Big\{ \sum_{m=1}^{\tau-1} c(Y_m) \geq \delta \frac {c_i {\big|\log \big( \frac 1 {\nu_i}{ \sum_{y \in \Y_{j(i)}}\overline{R}_{yi}} \big)\big|}} { l(i)}\Big\} \geq 1.
\end{align*}
\end{lemma}

\begin{proof}
Fix $\overline{R}$ such that $0 < \sum_{y \in \Y_{j(i)}}\overline{R}_{yi} < \nu_i$. Then we have $- \log ( \sum_{y \in \Y_{j(i)}}\overline{R}_{yi}/ {\nu_i} ) = |\log ( \sum_{y \in \Y_{j(i)}}\overline{R}_{yi}/ {\nu_i} )|$. Now in Lemma \ref{lemma_lower_bound_2}, set $j=j(i)$ and 
\begin{align*}
L := \sqrt{\delta} \frac {\big|\log \left( \frac 1 {\nu_i}{ \sum_{y \in \Y_{j(i)}}\overline{R}_{yi}} \right)\big|} {l(i)} \quad \textrm{and} \quad \gamma = \sqrt{\delta} c_i,
\end{align*}
and choose $k > 1$ such that $0 < k \sqrt{\delta} < 1$. Then we have
\begin{align*}
&\inf_{(\tau,d) \in \overline{\Delta}(\overline{R})} \P_i \left\{ \sum_{m=1}^{\tau-1} c(Y_m) \geq \delta \frac { c_i {\left|\log \left( \frac 1 {\nu_i}{ \sum_{y \in \Y_{j(i)}}\overline{R}_{yi}} \right)\right|}} {l(i)} \right\} \\
&\qquad \geq 1 - \frac {\sum_{y \in \Y \setminus \Y_{j(i)}} \overline{R}_{yi}} {\nu_i}  - \left( \frac  {\sum_{y \in \Y_{j(i)}}\overline{R}_{yi}} {\nu_i} \right)^{1-k \sqrt{\delta}} \\ & \qquad - \P_i \left\{ \sup_{n \leq \theta + L} \Lambda_n (i,j(i)) > kL l(i)\right\} - \P_i \left\{ \min_{n \geq L} \frac {\sum_{m=1}^{n-1} c(Y_m)} n < \gamma \right\}.
\end{align*}
This vanishes as $\overline{R}_i \downarrow 0$ because $0 < 1-k \sqrt{\delta} < 1$ and $\overline{R}_i  \downarrow 0 \Longrightarrow L \uparrow \infty$ and $0 < \gamma < c_i$.  Indeed, by \cite{Dayanik_2012}, Lemma A.2. for any $k > 1$, $\P_i \big\{ \sup_{n \leq \theta + L} \Lambda_n (i,j(i)) > k L l(i)\big\} \xrightarrow{L \uparrow \infty} 0$,
and because ${\sum_{m=1}^{n-1} c(Y_m)} / n$ converges $\P_i$-a.s.\ to $c_i$, $
\P_i \big\{ \min_{n \geq L} \frac {\sum_{m=1}^{n-1} c(Y_m)} n < \gamma \big\} \xrightarrow{L \uparrow \infty} 0$,
any $0 < \gamma < c_i$.
\end{proof}

%


\begin{proof}[Proof of Lemma \ref{lemma_lower_bound}]
Fix a set of positive constants $\overline{R}$, $0 < \delta < 1$ and $(\tau,d) \in \Delta$.
We have by Markov inequality
\begin{align*}
\E_i \left[ \frac {D_i^{(c,m)}(\tau)} {\left( { \frac {c_i} {l(i)} \left|\log \left( \frac 1 {\nu_i}{\sum_{y \in \Y_{j(i)}}\overline{R}_{yi}} \right) \right|} \right)^m}  \right] &\geq \delta
\P_i \left\{  \frac {\left(\sum_{m=1}^{\tau-1} c(Y_m)\right)^m} {\left( { \frac {c_i} {l(i)} \left|\log \left( \frac 1 {\nu_i}{\sum_{y \in \Y_{j(i)}}\overline{R}_{yi}} \right) \right|} \right)^m} \geq
\delta \right\} \\
&= \delta \P_i \left\{  \sum_{m=1}^{\tau-1} c(Y_m) \geq \delta^{\frac 1 m}
\frac {c_i {\left|\log \left( \frac 1 {\nu_i}{ \sum_{y \in \Y_{j(i)}}\overline{R}_{yi}} \right)\right|}} { l(i)}
\right\}.
\end{align*}
By taking infimum and then limits on both sides,
\begin{multline*}
\liminf_{\overline{R}_i \downarrow 0}\inf_{(\widetilde{\tau},\widetilde{d}) \in \overline{\Delta}(\overline{R})} \E_i \left[
\frac {D_i^{(c,m)}(\widetilde{\tau})} {\left( { \frac {c_i} {l(i)} \left|\log \left( \frac 1 {\nu_i}{\sum_{y \in \Y_{j(i)}}\overline{R}_{yi}} \right) \right|} \right)^m}
\right] \\ \geq \delta \liminf_{\overline{R}_i \downarrow 0} \inf_{(\widetilde{\tau},\widetilde{d}) \in
\overline{\Delta}(\overline{R})}\P_i \left\{ \sum_{m=1}^{\widetilde{\tau}-1} c(Y_m) \geq
\delta^{\frac 1 m} \frac {c_i {\left|\log \left( \frac 1 {\nu_i}{ \sum_{y \in \Y_{j(i)}}\overline{R}_{yi}} \right)\right|}} { l(i)}  \right\},
\end{multline*}
which is greater than or equal to $\delta$ by Lemma
\ref{lemma_for_lower_bound}. Therefore, the claim holds because $0 <
\delta < 1$ is arbitrary.
\end{proof}

\subsection{Proof of Lemma \ref{lemma_example_1_lambda}}
We first simplify $\widetilde{\alpha}_n^{(i)} (x_1,...,x_n)$ as in  \eqref{def_alpha_sum}.  Corresponding to the event that $Y$ is absorbed by $\Y_i$ at time $t \geq 0$, let the set of paths of $Y$ until time $n$ is given by $\mathcal{S}_{t,n}^{(i)}$ where
\begin{align*}
\mathcal{S}_{0,n}^{(i)} := \{ (i, \ldots, i)\} \quad \textrm{and} \quad \mathcal{S}_{t,n}^{(i)} := \{ (y_0, \ldots, y_n): y_0,\ldots, y_{t-1} \in \mathcal{Y}_{0}, \, y_t, \ldots, y_n = i \}, \quad 1 \leq t \leq n, 
\end{align*}
and by assumption
\begin{align*}
f(y_0, \cdot) \equiv \cdots \equiv f(y_{t-1}, \cdot)\equiv f_0(\cdot) \quad \textrm{and} \quad  f(y_t, \cdot) \equiv \cdots \equiv f(y_n, \cdot)  \equiv f_i(\cdot),  \quad y= (y_0, \ldots, y_n) \in \mathcal{S}_{t,n}^{(i)}. 
\end{align*}
\begin{lemma} \label{lemma_alpha_i}
For any $n \geq 1$ and $(x_1,\ldots x_n) \in E^n$,
\begin{align*}
\widetilde{\alpha}_n^{(i)} (x_1,...,x_n)   &= \left\{ \begin{array}{ll}  \prod_{l=1}^n
    f_0(x_l) \sum_{i \in \M} \nu_i \big(1- \sum_{t=0}^n \rho_t^{(i)} \big), & i = 0, \\ \nu_i \left[ \rho_0^{(i)}  \prod_{k=1}^{n}  f_i(x_k)  +  \sum_{t=1}^n  \rho_t^{(i)}\prod_{k=1}^{t-1}  f_0(x_k)  \prod_{k=t}^n  f_i(x_k) \right], & i \in \M. \end{array} \right.
\end{align*}
\end{lemma}
\begin{proof}
Because $\mathcal{S}_{0,n}^{(i)}, \mathcal{S}_{1,n}^{(i)}, \ldots,  \mathcal{S}_{n,n}^{(i)}$ are mutually disjoint and 
\begin{align*}
\mathcal{S}_{0,n}^{(i)} \sqcup \mathcal{S}_{1,n}^{(i)} \sqcup \cdots \sqcup  \mathcal{S}_{n,n}^{(i)} = \{ (y_0, \ldots, y_{n-1}, i): y_0, \ldots, y_{n-1} \in \mathcal{Y}_0 \cup \{i\} \} =: \mathcal{S}^{(i)}_n
\end{align*}
 or the set of paths under which $Y$ is in $\{i\}$ at time $n$,  and because $y_n = i$ for any $y$ in $\mathcal{S}_{0,n}^{(i)}, \ldots, \mathcal{S}_{n,n}^{(i)}$, we have
\begin{align*}
\widetilde{\alpha}_n^{(i)} (x_1,\ldots, x_n)  
&=\eta(i) \prod_{k=1}^{n}  f_i(x_k) + \sum_{t=1}^n\sum_{y \in \mathcal{S}_{t,n}^{(i)}}  \eta(y_0) \Big(\prod_{k=1}^{n} P(y_{k-1},y_k) f(y_k,x_k)\Big) \\
&=\eta(i) \prod_{k=1}^{n}  f_i(x_k)  +  \sum_{t=1}^n\sum_{y \in  \mathcal{S}_{t,n}^{(i)}} \Big( \eta(y_0) \Big(\prod_{k=1}^{t-1} P(y_{k-1},y_k) f(y_k,x_k) \Big) P(y_{t-1},i) f_i(x_t) \prod_{k=t+1}^n  f_i(x_k) \Big) \\
&= \eta(i) \prod_{k=1}^{n}  f_i(x_k)  +\sum_{t=1}^n \Big( \prod_{k=1}^{t-1}  f_0(x_k)  \prod_{k=t}^n  f_i(x_k) \Big) \sum_{y \in  \mathcal{S}_{t,n}^{(i)}} \Big( \eta(y_0) \prod_{k=1}^{t-1} P(y_{k-1},y_k)   P(y_{t-1},i) \Big)  \\
&=\nu_i \Big[ \rho_0^{(i)} \prod_{k=1}^{n}  f_i(x_k)  +  \sum_{t=1}^n  \rho_t^{(i)}\prod_{k=1}^{t-1}  f_0(x_k)  \prod_{k=t}^n  f_i(x_k) \Big],
\end{align*}
by \eqref{def_rho_t_i}.
On the other hand, 
\begin{multline*}
\widetilde{\alpha}_n^{(0)} (x_1,...,x_n)  = \sum_{y \in \mathcal{Y}_{0}}\alpha_n (x_1,...,x_n, y)  =  \prod_{l=1}^n
    f_0(x_l) \sum_{\mathcal{Y}^n \setminus (\cup_{i \in \mathcal{M}}\mathcal{S}_n^{(i)})}  \Big( \eta(y_0) \prod_{k=1}^{n-1} P(y_{k-1},y_k) \Big) P(y_{n-1},y) \\ = \prod_{l=1}^n
    f_0(x_l) \big(1- \sum_{i \in \M}\sum_{t=0}^n \nu_i \rho_t^{(i)} \big)  = \prod_{l=1}^n
    f_0(x_l) \sum_{i \in \M} \nu_i \big(1- \sum_{t=0}^n \rho_t^{(i)} \big).
\end{multline*}
\end{proof}

\begin{proof}[Proof of Lemma \ref{lemma_example_1_lambda}]
Fix $i \in \M$. By Lemma \ref{lemma_alpha_i},
\begin{align*}
\Lambda_n(i,0)  &= \log \left( \frac { \nu_i \big[ \rho_0^{(i)} \prod_{k=1}^n f_i(X_k) + \sum_{k=1}^n
    \rho_k^{(i)} \prod_{l=1}^{k-1} f_0(X_l) \prod_{m=k}^n f_i (X_m) \big] } {\prod_{l=1}^n
    f_0(X_l) \sum_{j \in \mathcal{M}} \nu_j \big(1- \sum_{t=0}^n \rho_t^{(j)} \big) } \right) \\
 &= \sum_{k=1}^n \log \frac { f_i(X_k)} { f_0(X_k)}  + \log \left( \frac { \nu_i \left[ \rho_0^{(i)} + \sum_{k=1}^n
    \rho_k^{(i)} \prod_{l=1}^{k-1} \frac {f_0(X_l)} {f_i(X_l)} \right] } {\sum_{j \in \M}\nu_j \big(1- \sum_{t=0}^n \rho_t^{(j)} \big) } \right) \\
&=  \sum_{k=1}^n \log \frac { f_i(X_k)} { f_0(X_k)}  + L_n^{(i)} + \log \nu_i  -  \log \big( \sum_{j \in \M} \nu_j \big(1- \sum_{t=0}^n \rho_t^{(j)} \big)  \big).
\end{align*}
For $j \in \M \backslash \{i\}$, we have
\begin{align*}
\Lambda_n(i,j) &= \log \frac {\nu_i} {\nu_j} +\log \left( \frac {  \rho_0^{(i)} \prod_{k=1}^n f_i(X_k) + \sum_{k=1}^n
    \rho_k^{(i)} \prod_{l=1}^{k-1} f_0(X_l) \prod_{m=k}^n f_i (X_m) } { \rho_0^{(j)} \prod_{k=1}^n f_j(X_k) + \sum_{k=1}^n
    \rho_k^{(j)} \prod_{l=1}^{k-1} f_0(X_l) \prod_{m=k}^n f_j (X_m) } \right) \\
&= \log \frac {\nu_i} {\nu_j} +\log \left( \prod_{k=1}^n \frac { f_i(X_k)} { f_j(X_k)} \frac {\exp L_n^{(i)}} {\exp L_n^{(j)}} \right) =  \log \frac {\nu_i} {\nu_j} +\sum_{k=1}^n \log \frac { f_i(X_k)} { f_j(X_k)}  + L_n^{(i)} - L_n^{(j)},
\end{align*}
and we can also write
\begin{align*}
\Lambda_n(i,j) 
&= \log \frac {\nu_i} {\nu_j} +\log \left( \frac 1 {\rho_n^{(j)}} \prod_{k=1}^n \frac { f_i(X_k)} { f_0(X_k)} \frac {\rho_0^{(i)}  + \sum_{k=1}^n
    \rho_k^{(i)} \prod_{l=1}^{k-1} \frac {f_0(X_l)} {f_i(X_l)} } {\frac {\rho_0^{(j)}} {\rho_n^{(j)}} \prod_{k=1}^n \frac {f_j(X_k)} {f_0(X_k)} + \sum_{k=1}^n
    \frac {\rho_k^{(j)}} {\rho_n^{(j)}} \prod_{m=k}^n \frac {f_j (X_m)} {f_0(X_m)} } \right) \\
&=  \log \frac {\nu_i} {\nu_j} -\log \rho_n^{(j)} + \sum_{k=1}^n \log \frac { f_i(X_k)} { f_0(X_k)}  + L_n^{(i)} - K_n^{(j)},
\end{align*}
as desired.
\end{proof}

\subsection{Proof of Lemma \ref{lemma_limit_l} }
The proof requires the following lemma, whose proof is similar to that of Lemma A.4 of \cite{Dayanik_2012}.

\begin{lemma}
  \label{lemma_convergence_series}
  Let $(\xi_n)_{n \geq 1}$ be a positive stochastic process and $T$
  an a.s.\ finite random time defined on the same probability space
  $(\Omega, \Ec, \P)$.  Given $T$, the random variables $(\xi_n)_{n\ge
    1}$ are conditionally independent, and $(\xi_n)_{1\le n \le T-1}$
  and $(\xi_n)_{n\ge T}$ have common conditional probability
  distributions $\P_\infty$ and $\P_0$ on $(\R, \mathscr{B}(\R))$, the
  expectations with respect to which are denoted by $\E_\infty$ and
  $\E_0$, respectively.  Suppose that $\E_\infty[\log \xi_1]$ and
  $\E_0[\log \xi_1]$ exist, and define
  \begin{gather}
    \label{eq:phi-psi-eta}
    \begin{gathered}
      \lambda := \E_0[\log \xi_1], \qquad \alpha := \E_{\infty} [
      \xi_1], \qquad \beta := \E_0 [ \xi_1], \qquad \gamma :=
      \max\{\alpha,\beta\}, \\
      \Phi_n:= \frac{1}{n} \log \prod^n_{k=1} \xi_k, \qquad 
      \psi_n := \log \Big(c+ \sum^{n}_{l=1} e^{l (\Phi_l + \delta_l)} \Big),
      \qquad \eta_n := \frac {\psi_n} n, \qquad n\ge 1
    \end{gathered}
  \end{gather}
  for some fixed constant $c > 0$ and deterministic sequence $\delta_l \xrightarrow{l \uparrow \infty} 0$. Then the following results hold under $\P$: 
  \begin{enumerate}
  \item[(i)]  We have $\eta_n \xrightarrow{n \uparrow \infty}
    \lambda_+$ a.s. 
  \item[(ii)]  If $\lambda < 0$, then the process $\psi_n$ converges as $n
    \uparrow \infty$ to a finite limit a.s. 
  \item[(iii)]  If $\gamma <\infty$, then $(|\eta_n|^r)_{n\ge 1}$ is
    uniformly integrable.
  \item[(iv)]  
    If $r\ge 1$ and $\max\{\E_{\infty}\left[|\log \xi_1|^r\right],
    \E_0\left[|\log \xi_1|^r\right]\}<\infty$, then
    $(|\Phi_n|^q)_{n\ge 1}$ is uniformly integrable for every $0\le q
    \le r$.
  \end{enumerate}
\end{lemma}

With this lemma, we prove Lemma \ref{lemma_limit_l}.  We first suppose $\varrho^{(j)} < \infty$. If in \eqref{eq:phi-psi-eta}, $\xi_k := e^{-\varrho^{(j)}} \frac {f_0(X_k)} {f_j(X_k)}$ and $c =\rho_0^{(j)} + \rho_1^{(j)} > 0$, 
\begin{align*}
L_n^{(j)} 
&= \log \Big( \rho_0^{(j)} + \sum_{k=1}^n
 {\rho_k^{(j)}}   e^{(k-1) (\Phi_{k-1} + \varrho^{(j)})}\Big)  = \log \Big( c+ \sum_{k=2}^n
  {\rho_k^{(j)}}   e^{(k-1) (\Phi_{k-1} + \varrho^{(j)})}\Big) \\ 
 &=  \log \Big(  c +   \sum_{k=1}^{n-1}
 {\rho_{k+1}^{(j)}}   e^{k (\Phi_k + \varrho^{(j)})} \Big) 
= \log \Big( c + \sum_{k=1}^{n-1}
\exp (\log \rho_{k+1}^{(j)} + k\varrho^{(j)})   e^{k \Phi_n } \Big) = \log \Big( c +  \sum_{k=1}^{n-1}
  e^{k (\Phi_k + \delta_k)} \Big) 
\end{align*}
where $\delta_k := (\log \rho_{k+1}^{(j)})/k +  \varrho^{(j)} \xrightarrow{k \uparrow \infty}  0$ by Assumption \ref{assumption_rho}.  
Given that $\mu=i$ and $\theta=t$ for any fixed $i\in \M$ and $t\ge
1$, the random variables $\xi_t, \xi_{t+1},\ldots$ are conditionally
i.i.d.\ with a common distribution independent of $t$; thus, the
change time $\theta$ plays the role of the random time $T$ in Lemma
\ref{lemma_convergence_series}.  Then by Lemma
\ref{lemma_convergence_series} (i) and (\ref{definition_q_i_j_0}) we
have $L_n^{(j)}/n \xrightarrow[n \uparrow \infty]{\P_i-a.s.}
\big[\int_E \big( -\varrho^{(j)} + \log \frac {f_0(x)}{f_j(x)}
\big) f_i(x) m (\diff x)\big]_+ = \big[ q(i,j)-q(i,0) - \varrho^{(j)}
\big]_+$, which proves (ii) immediately if $j\in \M\setminus\{i\}$,
and (i) and (iv) by Lemma \ref{lemma_convergence_series} (ii) if $j=i$
after noticing that  by (\ref{positiveness_kullback_leibler})
\begin{align*}
\int_E \left( -\varrho^{(i)} + \log \frac {f_0(x)}{f_i(x)}
\right) f_i(x) m (\diff x)= q(i,i) -q(i,0) - \varrho^{(i)} = -q(i,0) -
\varrho^{(i)} < 0.
\end{align*}

Similarly, if
$j \in \Gamma_i$, (v) holds by Lemma \ref{lemma_convergence_series}
(ii), since 
\begin{align*}
\int_E \left( -\varrho^{(j)} + \log \frac {f_0(x)}{f_j(x)}
\right) f_i(x) m (\diff x)= q(i,j) -q(i,0) - \varrho^{(j)} <
0
\end{align*}
 by the definition of $\Gamma_i$. By (\ref{definition_K}), the SLLN
and (ii),
\begin{multline*}
  \frac{1}{n} K_n^{(j)} = - \frac 1 n \log \rho_n^{(j)} + \frac{1}{n} \sum_{l=1}^{n \wedge (\theta -1)}
  \log \frac {f_j(X_l)}{f_0(X_l)} +
  \frac{1}{n} \sum_{l=\theta \land n}^n \log 
    \frac
    {f_j(X_l)}{f_0(X_l)}+ \frac{1}{n} L_n^{(j)}\\
  \xrightarrow[n \uparrow \infty]{\P_i-a.s.} \varrho^{(j)}  +0 -q(i,j) + q(i,0) + \left[q(i,j)-q(i,0) - \varrho^{(j)} \right]_+,
\end{multline*}
which equals $\left[ q(i,j)-q(i,0)-\varrho^{(j)}\right]_-$ and proves
(iii). For the proof of (vi), note that by Minkowski's inequality
\begin{multline*}
  \left|\frac{1}{n} L^{(j)}_n\right|^r = \left|\frac{\log (\rho_0^{(j)} + \rho_1^{(j)})}{n}
    + \frac{n-1}{n}\,\frac{\psi_{n-1}}{n-1}\right|^r \le
  \left(\left|\frac{\log (\rho_0^{(j)} + \rho_1^{(j)})}{n}\right| +
    \left|\frac{n-1}{n}\,\frac{\psi_{n-1}}{n-1}\right|\right)^r\\
  \le 2^{r-1}\left(\left|\frac{\log (\rho_0^{(j)} + \rho_1^{(j)})}{n}\right|^r +
    \left|\frac{n-1}{n}\right|^r
    \left|\frac{\psi_{n-1}}{n-1}\right|^r\right) \le
  2^{r-1}\left(\left|\frac{\log (\rho_0^{(j)} + \rho_1^{(j)})}{n}\right|^r +
    \left|\frac{\psi_{n-1}}{n-1}\right|^r\right).
\end{multline*}
Because $(|\log (\rho_0^{(j)} + \rho_1^{(j)})/n|^r)_{n\ge 1}$ is bounded, and according
to Lemma \ref{lemma_convergence_series} (iii) the process
$(|\psi_n/n|^r)_{n\ge 1}$ is uniformly integrable under $\P_i$ for
every $r\ge 1$ when (\ref{eq:sufficient-for-UI-of-L-average}) is
satisfied, we have (vi). Finally, for the proof of (vii),  (\ref{definition_K}) implies
\begin{align*}
  \Big| \frac{1}{n}K^{(j)}_n\Big|^r = \Big|- \frac 1 n  \log \rho_n^{(j)} - \varrho^{(j)} + \frac{1}{n}\log
    \prod^n_{k=1}  e^{\varrho^{(j)}} \frac{f_j(X_k)}{f_0(X_k)}
    +\frac{1}{n} L^{(j)}_n\Big|^r  \\
 \le 2^{r-1}\Big( \Big|\frac 1 n  \log \rho_n^{(j)} + \varrho^{(j)}\Big|^r + \Big|\frac{1}{n}\log \prod^n_{k=1}
      e^{\varrho^{(j)}}\frac{f_j(X_k)}{f_0(X_k)}\Big|^r +
    \Big|\frac{1}{n} L^{(j)}_n\Big|^r\Big).
\end{align*}
Because (\ref{eq:sufficient-for-UI-of-L-average}) holds,
$(|L^{(j)}_n/n|)_{n\ge 1}$ is uniformly integrable by (vi). If we set
$\xi_k := e^{\varrho^{(j)}}[f_j(X_k)/f_0(X_k)]$ for every $k\ge 1$ in
(\ref{eq:phi-psi-eta}), then (\ref{eq:sufficient-for-UI-of-K-average})
and Lemma \ref{lemma_convergence_series} (iv) imply that
$(|\frac{1}{n}\log \prod^n_{k=1}
(e^{\varrho^{(j)}}\frac{f_j(X_k)}{f_0(X_k)})|^r)_{n\ge 1}$ is uniformly
integrable. Therefore, $(|K^{(j)}_n/n|^r)_{n\ge 1}$ is uniformly
integrable, and the proof of (vii) is complete.

By Remark \ref{remark_regarding_l}  (1), it is now sufficient to prove (ii), (v) and (vi), when $\varrho^{(j)} = \infty$ (which implies $q(i,j) < \infty$ by Assumption \ref{condition_rho_q_finite}).  For any $M > q(i,j)-q(i,0)$,  $L_n^{(j)}$ is bounded by
\begin{multline*}
L_n^{(j,M)} := \log \Big( c + \sum_{k=2}^n
   (\rho_k^{(j)} \vee e^{-(k-1)M}) \prod_{l=1}^{k-1} \frac {f_0(X_l)} {f_j(X_l)} \Big) =\log \Big(  c +   \sum_{k=1}^{n-1}
 ({\rho_{k+1}^{(j)}} \vee  e^{-kM})  \prod_{l=1}^{k} \frac {f_0(X_l)} {f_j(X_l)}\Big) \\
=\log \Big(  c +   \sum_{k=1}^{n-1}
 ({\rho_{k+1}^{(j)}} e^{kM}\vee  1)  \prod_{l=1}^{k} e^{-M}\frac {f_0(X_l)} {f_j(X_l)}\Big)  =\log \Big(  c +   \sum_{k=1}^{n-1}
e^{k([M + (\log\rho_{k+1}^{(j)})/k]\vee  0)}  \exp \Big(\sum_{l=1}^{k} ({-M} + \log \frac {f_0(X_l)} {f_j(X_l)}) \Big) \Big). 
\end{multline*}
Because $[M + (\log\rho_{k+1}^{(j)})/k]\vee  0 \xrightarrow{k \uparrow \infty} 0$ by $\varrho^{(j)}=\infty$, applying Lemma 
\ref{lemma_convergence_series} (i) we obtain $L_n^{(j,M)} /n\xrightarrow[n \uparrow \infty]{\P_i-a.s.} 0$. Because  $L_n^{(j)}$ is bounded between $\log c$ and $L_n^{(j,M)}$, we obtain  $L_n^{(j)} /n\xrightarrow[n \uparrow \infty]{\P_i-a.s.} 0$ which proves (ii).  Because $L_n^{(j)}$ is increasing $\P_i$-a.s., its limit $L_\infty^{(j)}$ exists.  Moreover, because it is bounded by $L_\infty^{(j,M)} < \infty$,  $L_\infty^{(j)}$ is finite $\P_i$-a.s.\ or (v) holds.  Finally, because $L_n^{(j)}/n$ is bounded by $L_n^{(j,M)}/n$ and the latter is $L^r(\P_i)$-uniformly integrable, we also have (vi).

\subsection{Proof of Lemma  \ref{lemma_example_2_lambda}}

Fix $i \in \M$. Similarly to Lemma \ref{lemma_alpha_i},
for any $n \geq 1$ and $(x_1,\ldots, x_n) \in E^n$, we obtain
\begin{align*}
\widetilde{\alpha}_n^{(i)} (x_1,\ldots, x_n) 
&= \nu_i \Big[ \rho_0^{(i)} \prod_{k=1}^n f_i(x_k) + \sum_{k=1}^n
    \rho_k^{(i)} \prod_{l=1}^{k-1} f_i^{(0)}(x_l) \prod_{m=k}^n f_i(x_m) \Big], \\
\sum_{y \in \mathcal{Y}_{0}^{(i)}}\alpha_n (x_1,...,x_n, y) &= \nu_i \big(1- \sum_{t=0}^n \rho_t^{(i)} \big) \prod_{l=1}^n
    f_i^{(0)}(x_l).
\end{align*}
Therefore, for every $j \in \M \backslash \{i\}$, 
\begin{align*}
\Lambda_n(i,j) &= \log \left( \frac {\nu_i \big[ \rho_0^{(i)} \prod_{k=1}^n f_i(x_k) + \sum_{k=1}^n
    \rho_k^{(i)} \prod_{l=1}^{k-1} f_i^{(0)}(x_l) \prod_{m=k}^n f_i (x_m) \big]} {\nu_j \big[ \rho_0^{(j)} \prod_{k=1}^n f_j(x_k) + \sum_{k=1}^n
    \rho_k^{(j)} \prod_{l=1}^{k-1} f_j^{(0)}(x_l) \prod_{m=k}^n f_j (x_m) \big]} \right) \\
&= \log \left( \frac {\nu_i} {\nu_j} \prod_{k=1}^n \frac { f_i(x_k)} { f_j(x_k)} \frac {\exp L_n^{(i)}} {\exp L_n^{(j)}} \right) = \log \frac {\nu_i} {\nu_j} + \sum_{k=1}^n \log \frac { f_i(x_k)} { f_j(x_k)}  + L_n^{(i)} - L_n^{(j)},
\end{align*}
and
\begin{align*}
\Lambda_n(i,j) 
&= \log \left( \frac {\nu_i} {\nu_j} \frac 1 {\rho_n^{(j)}}\prod_{k=1}^n \frac { f_i(x_k)} { f_j^{(0)}(x_k)} \frac {\rho_0^{(i)}  + \sum_{k=1}^n
    \rho_k^{(i)} \prod_{l=1}^{k-1} \frac {f_i^{(0)}(x_l)} {f_i(x_l)} } {\frac {\rho_0^{(j)}} {\rho_n^{(j)}} \prod_{k=1}^n \frac {f_j(x_k)} {f_j^{(0)}(x_k)} + \sum_{k=1}^n
    \frac {\rho_k^{(j)}} {\rho_n^{(j)}} \prod_{m=k}^n \frac {f (j, x_m)} {f_j^{(0)}(x_m)} } \right) \\
&= \log \frac {\nu_i} {\nu_j} - \log \rho_n^{(j)} + \sum_{k=1}^n \log \frac { f_i(x_k)} { f_j^{(0)}(x_k)}  + L_n^{(i)} - K_n^{(j)}.
\end{align*}
On the other hand, for every $j \in \M$,
\begin{align*}
\Lambda_n^{(0)}(i,j) &= \log \left( \frac {\nu_i \big[ \rho_0^{(i)} \prod_{k=1}^n f_i(x_k) + \sum_{k=1}^n
    \rho_k^{(i)} \prod_{l=1}^{k-1} f_i^{(0)}(x_l) \prod_{m=k}^n f_i (x_m) \big]} {\nu_j \big(1- \sum_{t=0}^n \rho_t^{(j)} \big) \prod_{l=1}^n
    f_j^{(0)}(x_l)} \right) \\
 &= \sum_{k=1}^n \log \frac { f_i(x_k)} { f_j^{(0)}(x_k)}  + \log \left( \frac {\nu_i \Big[ \rho_0^{(i)} + \sum_{k=1}^n
    \rho_k^{(i)} \prod_{l=1}^{k-1} \frac {f_i(x_l)} {f_j^{(0)}(x_l)}  \Big]} {\nu_j \big(1- \sum_{t=0}^n \rho_t^{(j)} \big) } \right) \\
&= \log \frac {\nu_i} {\nu_j} + \sum_{k=1}^n \log \frac { f_i(x_k)} { f_j^{(0)}(x_k)}  + L_n^{(i)} -  \log \big( 1- \sum_{t=0}^n \rho_t^{(j)}  \big).
\end{align*}

\bibliographystyle{abbrvnat}

	\bibliography{thesis} 
\end{document}